\documentclass[a4paper,UKenglish,cleveref,nolineno,autoref]{socg-lipics-v2021}

\pdfoutput=1 

\title{Fast free resolutions of bifiltered chain complexes}

\author{Ulrich Bauer}{Department of Mathematics and MDSI and MCML, Technical University of Munich, Germany}{mail@ulrich-bauer.org}{orcid.org/0000-0002-9683-0724}{}

\author{Tamal K. Dey}{Department of Computer Science, Purdue University, USA}{tamaldey@purdue.edu}{https://orcid.org/0000-0001-5160-9738}{NSF grants
DMS-2301360 and CCF-2437030}

\author{Michael Kerber}{Institute of Geometry, Graz University of Technology, Austria}{kerber@tugraz.at}{https://orcid.org/0000-0002-8030-9299}{Austrian Science Fund (FWF) grant 10.55776/P33765}
\author{Florian Russold}{Institute of Geometry, Graz University of Technology, Austria}{russold@tugraz.at}{https://orcid.org/0009-0003-2978-0477}{Austrian Science Fund (FWF) grants 10.55776/P33765 and 10.55776/W1230}
\author{Matthias S\"ols}{Department of Mathematics, Technical University of Munich, Germany and Institute of Geometry, Graz University of Technology, Austria}{matthias.soels@tum.de}{Add orcid number}{Austrian Science Fund (FWF) grants 10.55776/P33765 and 10.55776/W1230}

\authorrunning{U. Bauer, T. Dey, M. Kerber, F. Russold, M. Söls}

\Copyright{TODO}

\ccsdesc{Mathematics of computing~Topology}
\keywords{Topological Data Analysis, Multi-Parameter Persistence, Multi-Critical Bifiltrations}

\category{} 

\relatedversion{} 

\supplement{The C++ library \textsc{multi-critical} is available at: \url{https://bitbucket.org/mkerber/multi_critical/src/main/}. The benchmark files are available upon request.  }

\acknowledgements{}

\nolinenumbers 

\EventEditors{John Q. Open and Joan R. Access}
\EventNoEds{2}
\EventLongTitle{42nd Conference on Very Important Topics (CVIT 2016)}
\EventShortTitle{CVIT 2016}
\EventAcronym{CVIT}
\EventYear{2016}
\EventDate{December 24--27, 2016}
\EventLocation{Little Whinging, United Kingdom}
\EventLogo{}
\SeriesVolume{42}
\ArticleNo{23}

\newcommand{\ignore}[1]{}
\newcommand{\define}[1]{\emph{#1}} 
\newcommand{\field}{\Bbbk}

\newcommand{\im}{\mathrm{im}\;}

\newcommand{\R}{\mathbb{R}}

\newcommand{\minus}{\mathbin{\scalebox{0.6}[0.9]{$-$}}}

\usepackage{tikz-cd}
\usetikzlibrary{arrows}

\usepackage{enumitem}

\usepackage[ruled,vlined]{algorithm2e}

\usepackage{thm-restate}
\usepackage{adjustbox}
\usepackage{mathtools}

\hideLIPIcs

\begin{document}

\maketitle

\begin{abstract}
In a $k$-critical bifiltration, every simplex enters along a staircase with at most $k$ steps. Examples with $k>1$ include degree-Rips bifiltrations and models of the multicover bifiltration. We consider the problem of converting a $k$-critical bifiltration into a $1$-critical (i.e. free) chain complex with equivalent homology. This is known as computing a free resolution of the underlying chain complex and is a first step toward post-processing such bifiltrations.

We present two algorithms. The first one computes free resolutions corresponding to path graphs and assembles them to a chain complex by computing additional maps. The simple combinatorial structure of path graphs leads to good performance in practice, as demonstrated by extensive experiments. However, its worst-case bound is quadratic in the input size because long paths might yield dense boundary matrices in the output. Our second algorithm replaces the simplex-wise path graphs with ones that maintain short paths which leads to almost linear runtime and output size.

We demonstrate that pre-computing a free resolution speeds up the task of computing a minimal presentation of the homology of a $k$-critical bifiltration in a fixed dimension. Furthermore, our findings show that a chain complex that is minimal in terms of generators can be asymptotically larger than the non-minimal output complex of our second algorithm in terms of description size.
\end{abstract}

\section{Introduction}
\subparagraph{Motivation and problem statement.}
Multi-parameter persistence is a branch of topological data analysis
where a data set (e.g., a point cloud) is filtered with respect to two or more parameters
and the topological evolution of the data when changing the parameters is analyzed.
In this context,
the first step of a computational pipeline for two parameters typically consists of the computation of a
\define{bifiltration} of simplicial complexes, that is, a family of simplicial complexes indexed by $\mathbb N^2$
that grows when increasing the parameters.
We refer to the parameter set $\mathbb N^2$ as the \define{grades} of the bifiltration.
A bifiltration can be equivalently described by
determining the \define{support} of each simplex, that is, the set of grades at which the
simplex is part of the complex.
A bifiltration is called \define{free} or
\define{$1$-critical} if the support
of every simplex is a \define{principal upset}, that is, the upward closure of a single element in the parameter space.
More generally, a bifiltration is \define{$k$-critical} if the support of a simplex
is the upward closure of at most $k$ elements; see Figure~\ref{fig:1_crit_k_crit} for a visualization.

\begin{figure}
  \includegraphics[width=6cm]{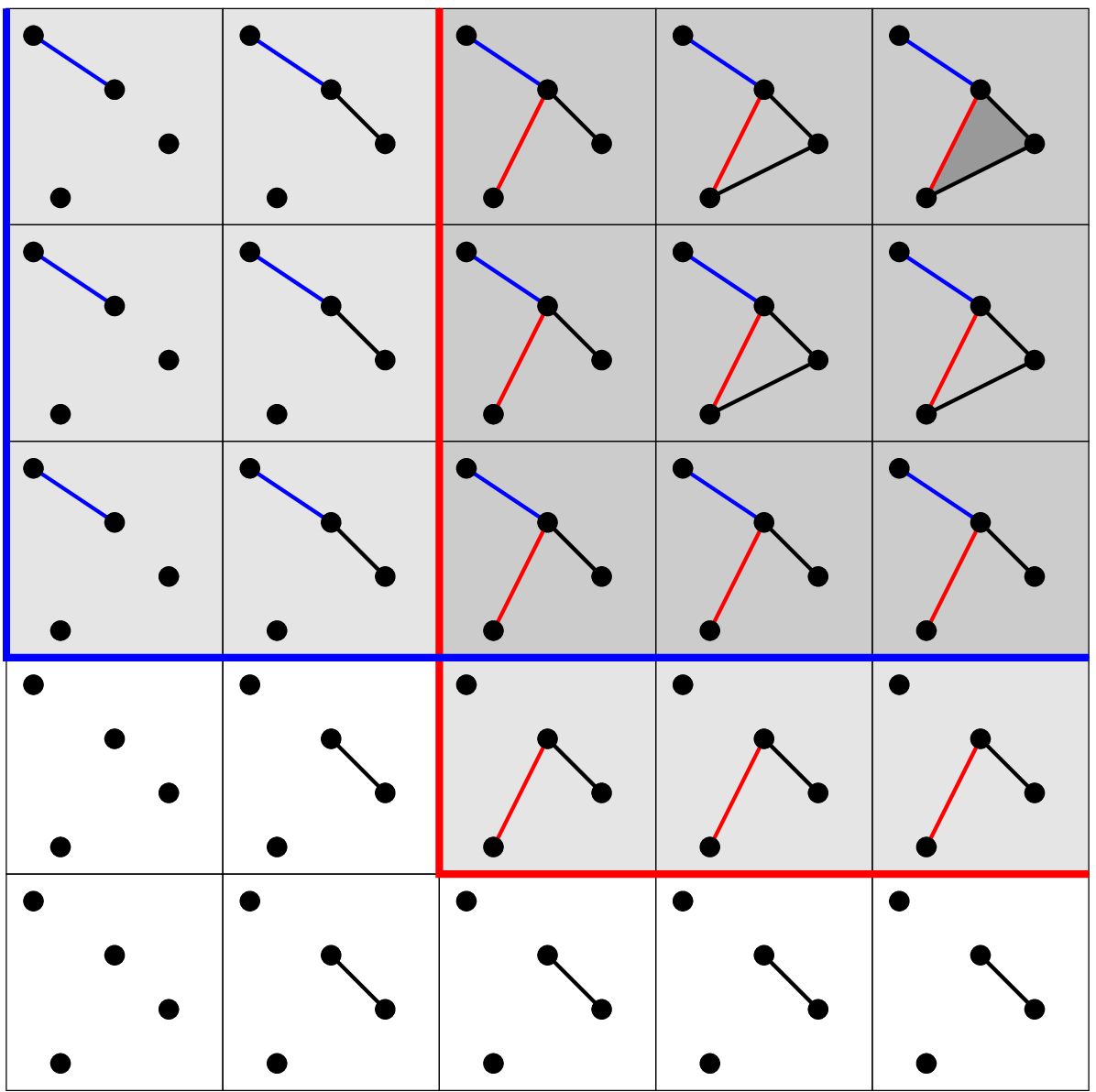}
  \hspace{1cm}
  \includegraphics[width=6cm]{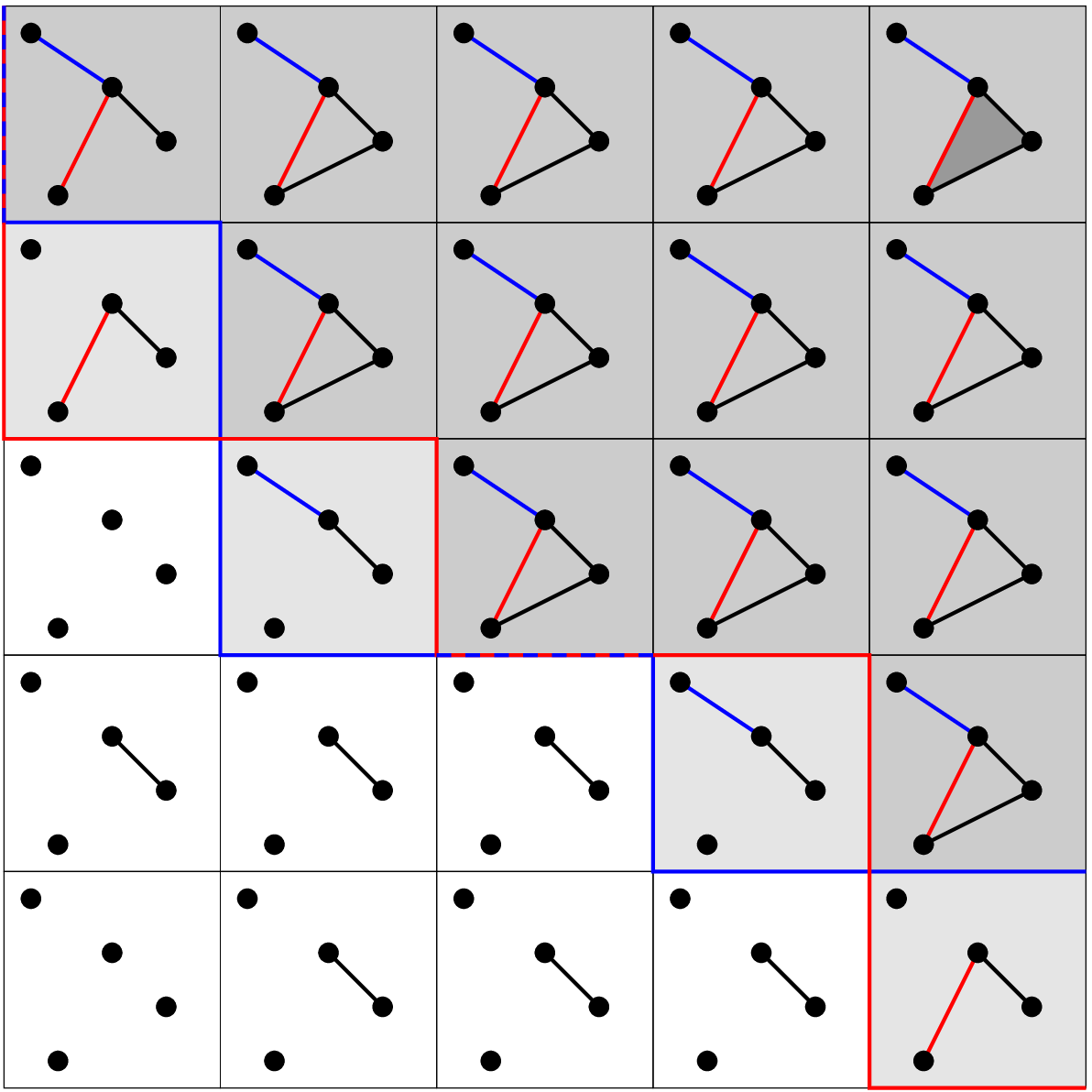}
  \caption{The bifiltration on the left is $1$-critical while the bifiltration on the right is $3$-critical. The support of every simplex is an upset, as visualized for two edges (red and blue, respectively).}
  \label{fig:1_crit_k_crit}
\end{figure}

A bifiltration gives rise to a bifiltered chain complex $C_\bullet$
\begin{equation*}
\begin{tikzcd}
0 & C_0 \arrow[l] & C_1 \arrow[l, "\partial_1"'] & \cdots \arrow[l] & C_{d-1} \arrow[l] & C_d \arrow[l, "\partial_d"']
\end{tikzcd}
\end{equation*}
with the $k$-simplices at a given grade forming the basis
of the $k$-chains at that grade. 
Free and $k$-critical filtered chain complexes are defined analogously.

Bifiltered chain complexes arise naturally in various constructions for topological data analysis, and several constructions produce bifiltrations that are not free.
The most prominent example is given by the degree-Rips bifiltration~\cite{lw-rivet,rolle-socg,rs-stable,bl-stability}.
Other examples
are the approximate multi-cover bifiltration~\cite{alonso-sparse} and Delaunay core bifiltration~\cite{Core_bifilt}.
On the other hand, free bifiltrations are most suitable for algorithmic and computational purposes.
There are fast algorithms for minimizing a bifiltered chain complex without changing its homology~\cite{fk-chunk,fkr-compression,ms-computing}
as well as for computing the homology of a chain complex in terms of a minimal presentation~\cite{lw-computing,kr-fast,fkr-compression,bll-efficient}, 
but both approaches require free chain complexes as input.

In homological algebra, a standard technique is to replace a general chain complex by a free one, connected to the original chain complex by a chain map that induces an isomorphism in homology (a \define{quasi-isomorphism}).
The free chain complex together with the chain map is called a \define{free resolution} of the original complex.
The problem studied in this paper is how to efficiently compute a free resolution of a non-free bifiltered chain complex. 

At the homology level, one way to address the case of a non-free bifiltration was proposed by Chacholski,
Scolamiero and Vaccarino~\cite{swedish}: given a segment  $C:C_{m+1}\to C_m\to C_{m-1}$ of a chain complex, the authors describe an algorithm to compute a free chain complex
$F:F_{m+1}\to F_m\to F_{m-1}$ such that $H_m(C)$ and $H_m(F)$ are isomorphic, that is, a \define{free implicit representation}.
This algorithm suffices if one is interested in a presentation
of $H_m(C)$ for further processing. Methods to compute a projective implicit representation from families of simplicial complexes and general simplicial maps have been developed in \cite{dey2025computingprojectiveimplicitrepresentations,dey_et_al:LIPIcs.SoCG.2024.51}, where, as in our approach, graph theoretic methods are used to speed up computations.

There are good reasons to work on the level of entire chain complexes instead.
First of all, there are potential computational advantages, especially if one is interested in multiple homology dimensions (see Section \ref{sec:experiments}). 
Moreover, the chain complex structure can encode subtle information on the data that is lost at the homology level: for example, two chain complexes may have isomorphic homology in every dimension without being quasi-isomorphic. 

Computing free resolutions of general chain complexes is a standard task in computational algebra, available in the computer algebra software \textsc{Macaulay2} for a much wider range of chain complexes.
In the context of applied topology, we are mostly interested in very large bifiltered chain complexes with millions of generators but a simple combinatorial structure (i.e., simplicial boundary maps).
The goal of this paper is to develop specialized and highly optimized algorithms for this type of input data, which the general purpose algorithms implemented in existing computer algebra systems are not tailored for.

\subparagraph{Contributions.}
Our main contribution is to propose two algorithms to compute a free resolution of a bifiltered chain complex. 
Both algorithms rely on the same simple idea of expanding any $k$-critical simplex in dimension $p$ into a sequence of $k$ free copies, with consecutive copies related by $(p+1)$-dimensional elements at the join of their grades.
In algebraic terms, this corresponds to a free resolution of the upset module associated to the simplex.
In order to construct a valid total complex,
the algorithms introduce further maps to establish the chain complex property while maintaining quasi-isomorphism to the original filtered complex.
Finding these maps is computationally inexpensive and takes place at a purely combinatorial level.

The algorithms differ mainly in the choice of free resolution of the upset modules. 
In the first algorithm, the \define{path algorithm}, the free resolution corresponds to the chain complex of a filtered path graph (as a simple special case of a cellular resolution \cite{bayer_sturmfels_cellular,miller_sturmfels_book}).
We give an example of a family of simplicial bifiltrations so that the resulting free chain complex has a dense boundary matrix (for any choice of basis).
This example shows tightness of the worst-case runtime $O(n^2)$, where $n$ is the number of input generators (description size).

The second algorithm, the \define{log-path algorithm}, extends the path algorithm by adding additional relations to the free resolution of a simplex, such that any pair of copies of a simplex is connected via a sequence of relations of logarithmic length.
This ensures sparsity of the boundary matrices in the output chain complex, but also requires adding further higher relations (\define{syzygies}). 
Again, further maps are required to establish the chain complex property of the resulting total complex, allowing the algorithm to maintain sparsity and obtain a resolution with worst-case run time in $O(n\log^2n)$.

Our findings lead to  an interesting observation: as shown by our worst case example, a minimal resolution may require dense boundary matrices, while a non-minimal resolution may actually admit a sparse matrix representation, with asymptotically fewer non-zero entries ($O(n\log^2n)$ instead of $O(n^2)$, where $n$ is the description size of the complex).
This observation suggests that minimizing chain complexes does not necessarily speed-up subsequent algorithmic tasks, at least in certain worst-case examples.

We provide implementations of the path and log-path algorithms, in addition to the Chacholski--Scolamiero--Vaccarino algorithm for computing free implicit representations of homology.
Systematic tests on various $k$-critical bifiltrations show that the overhead of the log-path algorithm over the path algorithm does not exceed a factor of 3 in run time and a factor of 2 in the number of non-zero entries for our examples, while showing the expected improvement on the mentioned worst-case examples. 
Subsequently minimizing the free chain complex yields a further significant reduction of the size.
Furthermore, we consider the task of computing minimal presentations of homology in all degrees, comparing the approach of first computing a ``global'' free resolution with the approach of computing free implicit representations.
Our results show a clear computational advantage for the global approach.
Remarkably, for some instances, computing \emph{all} minimal
presentations using a free resolution is faster than computing a \emph{single} minimal presentation using the Chacholski--Scolamiero--Vaccarino algorithm.

\section{Bifiltered chain complexes}
\label{sec:Background}

\subparagraph{Bifiltrations.} A \define{simplicial bifiltration} $\mathcal{K}$ is an abstract simplicial complex together with a collection of subcomplexes $(\mathcal{K}_s)_{s\in \mathbb{N}^2}$ such that $\mathcal{K}_s\subseteq \mathcal{K}_t$ whenever $s\leq t$ (which means that $s_1\leq t_1$ and $s_2\leq t_2$).
As shown in Figure \ref{fig:1_crit_k_crit}, each simplex $\sigma$ enters $\mathcal{K}$ along a staircase bounding the \define{support} of $\sigma$, denoted $\mathrm{supp}(\sigma)$, which is an \define{upset} (an upward closed subset of $\mathbb{N}^2)$.
There is a unique minimal set of grades
$\mathcal{G}(\sigma):=\{x_1,\dots,x_m\}\subseteq \mathbb{N}^2$ 
such that $\mathrm{supp}(\sigma)=\{ s\in \mathbb{N}^2\mid \exists x_i\leq s\}$.
If $m\leq k$ for each $\sigma \in \mathcal{K}$, then the bifiltration is \define{$k$-critical}.

\subparagraph{Bipersistence modules.}

Simplicial bifiltrations give rise to bifiltered chain complexes. 
We first describe the elementary building blocks.
Each $p$-simplex $\sigma$ has a support $\mathrm{supp}(\sigma)$ with minimal generating set $\mathcal{G}(\sigma)=\{x_1,\dots,x_m\}$,
determining an \define{upset module} $U_{\{x_1,\dots,x_m\}}$ (which we also denote by $U_\sigma$ for brevity) given by 
\begin{equation*}
\begin{aligned}
(U_\sigma)_s &:= 
\begin{cases}
\field, & \text{if } \exists\, x_i \leq s, \\[6pt]
0 & \text{otherwise}
\end{cases}
&
\quad \text{and}\qquad
(U_\sigma)_{s, t} &= 
\begin{cases}
\mathrm{id}_\field, & \text{if } \exists\, x_i \leq s, \\[6pt]
0 & \text{otherwise}
\end{cases}
\end{aligned}
\end{equation*}
as illustrated in Figure \ref{fig:single_simplex}.
%
This is a special case of a \define{bipersistence module} $M$, which is a family $(M_s)_{s\in\mathbb{N}^2}$ of vector spaces over a field $\field$ together with a family of homomorphisms $(M_{s,t}\colon M_s \rightarrow M_t)_{s\leq t\in \mathbb{N}^2}$.
Although our results extend to arbitrary fields straightforwardly, we will stick to the case $\field=\mathbb{Z}_2$ in order to simplify the exposition.
A \define{morphism of bipersistence modules} $\varphi:M\rightarrow N$ is a natural transformation, that is, a family of linear maps $(f_s)_{s\in \mathbb{N}^2}$ that commute with the structure maps of $M$ and $N$: $f_t \circ M_{s,t} = N_{s,t}\circ f_s$ for $s\leq t$. 

A direct sum of upset modules, each generated by a single grade, is called \define{free}.
A \define{basis} of a free bipersistence module $F$ is a set of elements $\{b^{x_1},\dots,b^{x_n}\}$
where each $b^{x_i} \in F_{x_i}$, such that every $v \in F_y$, for any grade $y$, can be written in unique way as a linear combination of images of the basis elements under the structure maps (which are inclusions into the vector space $\bigcup_z F_z$).
Note that the grades appearing in this linear combination with a non-zero coefficient must be less or equal to $y$.
A morphism $f$ between free modules is then determined by the images of the basis elements for the domain, and can  be encoded by a matrix $[f]$.

\begin{figure}
\centering
  \includegraphics[width=8cm]{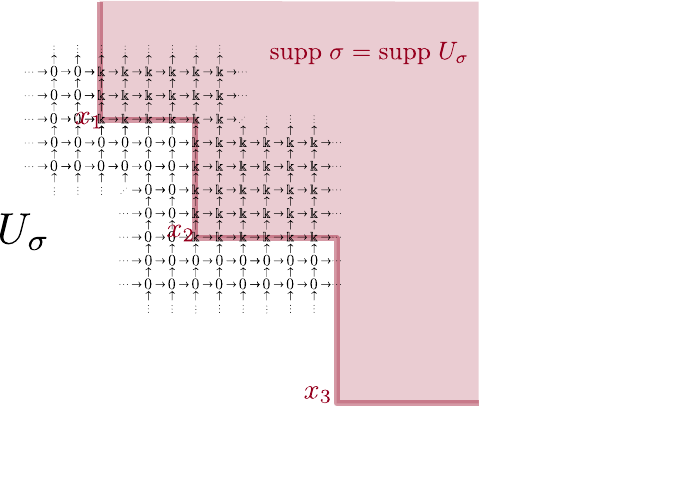}
  \caption{The upset module $U_\sigma$ induced by the simplex $\sigma$.}  
  \label{fig:single_simplex}
\end{figure}

\subparagraph{Chain complexes.} 
The upset modules associated to the simplices assemble to 
\begin{equation*}
\begin{tikzcd}
0 & C_0 \arrow[l] & C_1 \arrow[l,swap,"\partial_{1}"] & \cdots \arrow[l,swap,"\partial_2"] & C_{i-1} \arrow[l,swap,"\partial_{i-1}"] & C_{i} \arrow[l,swap,"\partial_{i}"] & C_{i+1} \arrow[l,swap,"\partial_{i+1}"] & \cdots \arrow[l]
\end{tikzcd},
\end{equation*}
which is a \define{bifiltered chain complex} with $C_i:= \bigoplus_{\sigma \in \mathcal{K}^{(i)}} U_\sigma$
and \define{boundary maps} $\partial_i:C_i \rightarrow C_{i-1}$ inherited from the simplicial complex $\mathcal{K}$, satisfying $\partial_{i}\circ \partial_{i+1}=0$. The notion of a $k$-critical bifiltered chain complex is defined analogously to the setting of a simplicial bifiltration.

Even though a bifiltered chain complex has a simple combinatorial structure, a \define{free chain complex} (a chain complex $F_\bullet$ with each $F_i$ free) is preferred in computational (and algebraic) settings.
This motivates the goal of finding a \define{free resolution}, that is, a free chain complex together with a chain map $F_\bullet\rightarrow C_\bullet$ that induces an isomorphism $H_i(F_\bullet)\cong H_i(C_\bullet)$ on homology for each $i\geq 0$.
Such a chain map is called a \define{quasi-isomorphism}.

\subparagraph{Data representation.} The list of all simplices of a bifiltration $\mathcal{K}$, each with a lexicographically ordered list of the minimal generating set of its support, and a list of its facets gives a full combinatorial description of $\mathcal{K}$ and will serve as the \define{input} to our algorithms.
This data representation generalizes to regular cell complexes, and in our case to bifiltered chain complexes where $\field=\mathbb{Z}_2$.
For the simplicial case, the number of facets is constant for each simplex dimension, ensuring sparsity.
Therefore the description size of the input is given, up to a constant factor, by the total number of generators in this list. 
Our \define{output} is a free resolution $F_\bullet$ of our input chain complex $C_\bullet$, again represented in the same data format.

\section{Free implicit representations of homology}
\label{sec:FI-reps}

\subparagraph{FI-reps.} To compute the homology $H_m(C_\bullet)$ of a $k$-critical chain complex in a chosen dimension $m$, Chacholski, Scolamiero, and Vaccarino \cite{swedish} provide a procedure
to construct a free chain complex segment $X \xleftarrow{f} Y \xleftarrow{g} Z $ from the input segment
\begin{equation}
\label{diag:chain_complex_segment}
C_{m-1}\xleftarrow{\partial_m} C_m \xleftarrow{\partial_{m+1}}C_{m+1}.
\end{equation}
such that $H_m(C_\bullet)\cong \ker f / \im g$. The pair of graded matrices $([f],[g])$ is called a \define{free implicit representation (FI-rep)} of $H_m(C_\bullet)$ \cite{lw-computing}, and it serves as input for the computation of minimal presentations of $H_m(C_\bullet)$ \cite{lw-computing,kr-fast}.

\subparagraph{The Chacholski--Scolamiero--Vaccarino algorithm.} We consider an input complex $C_\bullet$ induced by a simplicial bifiltration $\mathcal{K}$.
The first step is to extend each $U_\sigma$ in $C_{m-1}$, induced by an $(m-1)$-simplex $\sigma$, to a free upset module $U_{\{0\}}$.
The resulting free module $D_{m-1}\coloneqq 
\bigoplus_{\sigma\in \mathcal{K}^{(m-1)}} U_{\{0\}}$ , whose basis we denote by $\{b_\sigma^0\}_{\sigma\in \mathcal{K}^{(m-1)}}$, contains $C_{m-1}$ as a submodule, since the support of $U_\sigma$ is contained in the support of $U_{\{0\}}$.
Thus, postcomposing $\partial_m$ with this submodule inclusion does not affect the kernel: we have $\ker\partial_{m} = \ker\iota\circ\partial_{m}$.

As a second step, we \define{cover} each $U_\sigma$ in $C_{m+1}$, induced by an $(m+1)$-simplex $\sigma$ by a free bipersistence module.
This means that we replace each upset module $U_\sigma$ by the free module 
$
G_\sigma\coloneqq 
\bigoplus_{x_i \in \mathcal{G}(\sigma)} U_{\{x_i\}}
$,
connected to the upset module by a surjection $G_\sigma \twoheadrightarrow U_\sigma$ since the support of $G_\sigma$ equals the support of $U_\sigma$.
We call the basis elements $g^{x_i}_\sigma$ the \define{generators} of $
U_\sigma$.
The module $G_{m+1}$ arises from $C_{m+1}$ by covering $U_\sigma$ by $G_\sigma$ for each $(m+1)$-simplex via a canonical surjection \begin{tikzcd}G_{m+1}\arrow[r,two heads,"\alpha_{m+1}"] & C_{m+1} \end{tikzcd}.
Hence, precomposing $\partial_{m+1}$ with $\alpha_{m+1}$ does not affect the image.
The input \eqref{diag:chain_complex_segment} can now be replaced by \eqref{eq:segment-step2} which has isomorphic homology.
\begin{equation}
\label{eq:segment-step2}
\begin{tikzcd}
D_{m-1} &[10pt] C_m \arrow[l,swap,"\iota\circ\partial_{m}"] &[20pt] G_{m+1} \arrow[l,swap,"\partial_{m+1}\circ \alpha_{m+1}"] ,
\end{tikzcd}
\end{equation}

As the third and final step it remains to replace $C_m$. 
We first construct $G_m$ analogously to $G_{m+1}$, by substituting each $U_\sigma$ in $C_m$ by $G_\sigma$.
Precomposing $\iota\circ\partial_{m}$ with the surjection \begin{tikzcd}G_{m}\arrow[r,two heads,"\alpha_{m}"] & C_{m} \end{tikzcd} yields a map $\gamma\coloneqq \iota\circ\partial_m\circ\alpha_m\colon G_m\rightarrow D_{m-1}$, as depicted in \eqref{eq:csv_main_diag}.
For $\sigma\in \mathcal{K}^{(m)}$ with $\partial\sigma=\tau_0+\cdots +\tau_m$, the resulting map $\gamma$ sends a generator $g_\sigma^x$ in $G_m$ to $b^0_{\tau_0}+\cdots + b^0_{\tau_m}$ in $D_{m-1}$.
\begin{equation} \label{eq:csv_main_diag}
\begin{tikzcd} 
& & R_{m} \arrow[d,"p^1_m"] \\
& & G_m \arrow[d,"\alpha_{m}"] \arrow[dll,swap,dashed,yshift=2pt,"\gamma"]  & G_{m+1} \arrow[d,"\alpha_{m+1}"] \arrow[l,swap,dashed,"f_{m+1}^0"] \\
D_{m-1} & C_{m-1} \arrow[l,swap,"\iota"] & C_m \arrow[l,swap,"\partial_{m}"] & C_{m+1} \arrow[l,swap,"\partial_{m+1}"]
\end{tikzcd}
\end{equation}
                            
Replacing $C_m$ by $G_{m}$ makes it necessary to also replace the map $\partial_{m+1}\circ \alpha_{m+1} \colon G_{m+1}\rightarrow C_m$ by a map $f_{m+1}^0\colon G_{m+1}\rightarrow G_m$, representing $\partial_{m+1}$ on the generators and also making the diagram commute.
This map $f_{m+1}^0$ sends each generator $g_\sigma^x$ of an $(m+1)$-simplex $\sigma$ with $\partial\sigma=\tau_0+\cdots +\tau_m$ to $f_{m+1}^0(g_\sigma^x)=g_{\tau_0}^{y_0}+\cdots + g^{y_m}_{\tau_m}$, where each $g_{\tau_i}^{y_i}$ is a generator of $\tau_i$ chosen such that its grade satisfies $y_i\leq x$; such a $y_i$ always exists because the faces $\tau_i$ of a simplex $\sigma$ are present in the bifiltration whenever $\sigma$ itself is present.
The map $f_{m+1}^0$ thus makes the square in \eqref{eq:csv_main_diag} commute and is therefore called a \define{lift} of $\partial_{m+1}$.
See Figure \ref{fig:f0} for an illustration.
\begin{figure}
\centering
  \includegraphics[scale=0.45]{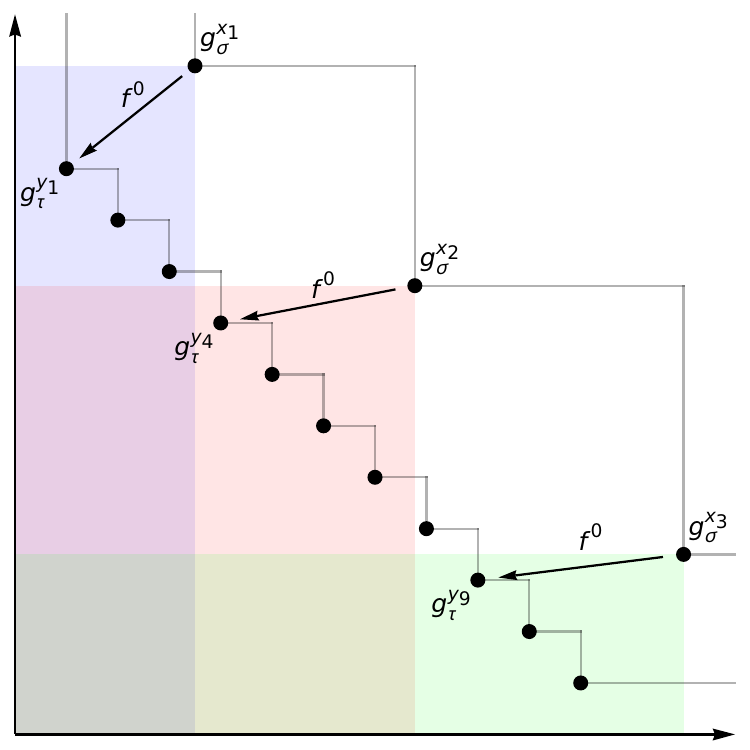}
  \caption{$f_{m+1}^0$ sends each generator $g_\sigma^x$ to a generator $g_\tau^y$ of the facet $\tau$ of $\sigma$ whose grade $y$ is in the downset of $x$ (say, with the smallest first coordinate).}  
  \label{fig:f0}
\end{figure}

Note that the kernel of $\gamma$ may not be isomorphic to the kernel of $\partial_m$; the surjection $\alpha_m$ maps each generator $g_\sigma^x$ of an $m$-simplex in $G_m$ to the same potential cycle in $C_m$, thus increasing the dimension of the kernel.
This can be resolved by relating these generators appropriately. Assume that the generators $g^{x_1}_\sigma,\ldots,g^{x_k}_\sigma$ of $U_\sigma$ are ordered w.r.t.\ the first coordinate of their grades. To represent $U_\sigma$ correctly, two consecutive generators $g^{x_{i}}_\sigma$ and $g^{x_{i+1}}_\sigma$ have to be identified at the join $y_i=x_i\vee x_{i+1}$ of their grades by a \define{relation} $r_\sigma^{y_i}$. With the free bipersistence module $R_\sigma\coloneqq 
\bigoplus_{y_i} U_{\{y_i\}}
$ with basis $\{r_\sigma^{y_i}\}_i$, we can present $U_\sigma$ by generators and relations via
\begin{equation}
\label{diag:free_res}
\begin{tikzcd}
0 & \arrow[l] U_\sigma & G_\sigma \arrow[l,swap,"\alpha_\sigma"] & R_\sigma , \arrow[l,swap,"p^1_\sigma"] 
\end{tikzcd}
\end{equation}
where $p^1_\sigma(r_\sigma^{y_i})=g_\sigma^{x_i}+g_\sigma^{x_{i+1}}$.
Since $p^1_\sigma$ is injective, $\alpha_\sigma$ is surjective, and $\ker \alpha_\sigma=\im p^1_\sigma$, the sequence in \eqref{diag:free_res} is a short exact sequence.
This means that $(G_\sigma \xleftarrow{p^1_\sigma}R_\sigma,\alpha_\sigma)$ already determines a \define{free resolution} of $U_\sigma$, as defined in Section \ref{sec:Background}, where $U_\sigma$ is considered as a chain complex concentrated in degree $0$. 

The free resolution in Diagram \ref{diag:free_res} has a simple combinatorial structure, given by a path graph $\mathcal{P}_\sigma$ with $k$ vertices corresponding to the generators $g_\sigma^{x_i}$ and $k-1$ edges $\{g_\sigma^{x_i},g_\sigma^{x_{i+1}}\}$ corresponding to the relations $r_\sigma^{y_i}$.
Note that vertices and edges implicitly carry the grades of the generators and relations.
The morphism $p^1_\sigma$ is represented by the (graded) incidence matrix of the graph.
We call this free resolution the \define{path resolution} of $U_\sigma$, see Figure \ref{fig:path_resolution}.
Note that a path resolution is a special case of a \define{cellular resolution} \cite{bayer_sturmfels_cellular,miller_sturmfels_book}.

\begin{figure}
\centering
  \includegraphics[width=8cm]{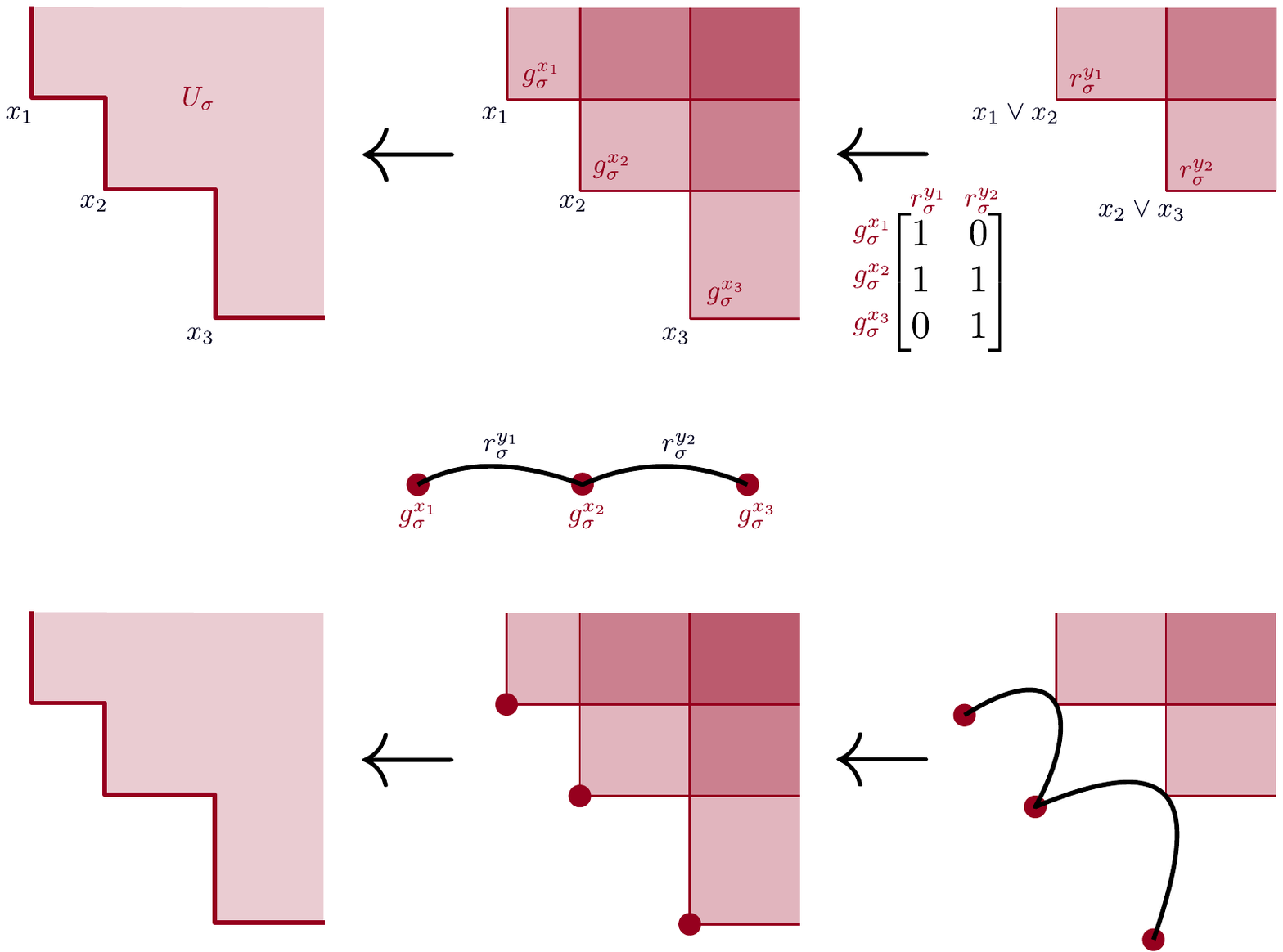}
  \caption{The top row illustrates the free resolution of $U_\sigma$ as in Diagram \ref{diag:free_res} and the middle row shows its path resolution. We often visualize both at once as in the bottom row. }  
  \label{fig:path_resolution}
\end{figure}

We now use the free resolution to replace the map $ \partial_{m+1}\circ \alpha_{m+1}$ in \eqref{eq:segment-step2}, assembling the following output complex from \eqref{eq:csv_main_diag}

\begin{equation} \label{diag:CSV_segment}
\begin{tikzcd}[column sep=large,row sep=large,ampersand replacement=\&,every label/.append style = {font = \small}]
D_{m-1} \& G_m \arrow[l,swap,"\gamma_m"] \&[25pt] G_{m+1}\oplus R_m \arrow{l}[swap]{\begin{pmatrix}f^0_{m+1} & p^1_m\end{pmatrix}}.
\end{tikzcd}
\end{equation}
By construction, the homology of this complex segment in degree $m$ is isomorphic to that of \eqref{eq:segment-step2}, and hence to that of the input complex.

\section{Path algorithm}
\label{sec:path}

\subparagraph{Description.} We now turn to the problem of finding a \define{free resolution}, that is, a chain complex of free modules that is quasi-isomorphic to the input $C_\bullet$, thus preserving homology in \textit{all} dimensions.
For this problem, extending the codomain of $\partial_i:C_i \rightarrow C_{i-1}$ to a free module generated in degree $0$, as done in the Chacholski--Scolamiero--Vaccarino algorithm, is not feasible anymore, since this changes $H_{i-1}(C_\bullet)$. Instead, we now carry out the step of 
substituting $C_i$ by its free resolution $G_i \xleftarrow{p^1_i} R_i \leftarrow 0 \leftarrow \cdots $ 
 in every dimension $i$ and
replacing $\partial_{i+1}$ by the (zeroth) lifts $f^0_{i+1}\colon G_{i+1}\rightarrow G_{i}$, 
yielding maps $\begin{pmatrix}f^0_{i+1} & p^1_i\end{pmatrix}$
as illustrated in \eqref{diag:path_step1}.

\begin{equation}
\label{diag:path_step1}
\adjustbox{max width=0.96\textwidth}{
\begin{tikzcd}[ampersand replacement=\&]
\& \vdots \arrow[d,swap,"\partial_3"] \& \vdots \arrow[d,swap,"f_3^0"] \&\& \&\& \\
0 \& C_2 \arrow[l] \arrow[d,swap,"\partial_2"] \& G_2 \arrow[d,swap,"f_2^0"] \arrow[l,swap,"\alpha_2"] \& R_2 \arrow[l,swap,"p^1_2"] \& 0 \arrow[l] \\
0 \& C_1 \arrow[l] \arrow[d,swap,"\partial_1"] \& G_1 \arrow[d,swap,"f_1^0"] \arrow[l,swap,"\alpha_1"] \& R_1 \arrow[l,swap,"p^1_1"] \& 0 \arrow[l] \arrow[dddrr,bend left,violet,shorten=40pt] \\
0 \& C_0 \arrow[l] \arrow[d] \& G_0 \arrow[d] \arrow[l,swap,"\alpha_0"] \& R_0 \arrow[l,swap,"p^1_0"] \& 0 \arrow[l] \\
\& 0 \& 0 \\[10pt]
\& \& \& 0 \& G_0 \arrow[l] \& G_1\oplus R_0 \arrow{l}[swap]{\begin{pmatrix}f_1^0 & p^1_0 \end{pmatrix}} \& G_2\oplus R_1 \arrow{l}[swap]{\begin{pmatrix}f_2^0 & p^1_1 \\ 0 & 0\end{pmatrix}} \& \cdots \arrow{l}[swap]{\begin{pmatrix}f_3^0 & p^1_2 \\ 0 & 0\end{pmatrix}} \& \phantom{a}
\end{tikzcd}}
\end{equation}

However, the sequence of maps, arising from this construction, does not yield a chain complex.
Given a relation $r_\sigma^{y_j}$ in $R_i$, we obtain 
\begin{equation*}
f^0_{i}\circ p^1_i(r_\sigma^{y_j})=f^0_{i}(g^{x_j}_\sigma)+f^0_{i}(g^{x_{j+1}}_\sigma)=g_{\tau_0}^{z_1}+\cdots+g_{\tau_i}^{z_i}+g_{\tau_0}^{z'_1}+\cdots+g_{\tau_i}^{z'_i}.
\end{equation*}
This term (also the term $f_{i-1}^0\circ f_i^0$) may not vanish, since the generators of $\sigma$ may be sent to different generators of its boundary simplices $\tau_j$ by $f^0_{i}$, i.e., $g_{\tau_j}^{z_j}\neq g_{\tau_j}^{z'_j}$.
However, it vanishes modulo relations, in the sense that its representative vanishes in $G_{i-1}/\im p^1_{i-1}$, see Lemma \ref{lem:path-lemma}.

\begin{lemma}[Path Lemma]
  \label{lem:path-lemma}
  For any $g_\tau^{x_j}$, $g_\tau^{x_\ell}$ in $G_{i-1}$, there is a unique set of distinct relations $r_\tau^{w_j},\ldots,r^{w_{\ell-1}}_\tau$ in $R_{i-1}$ such that $p^1_{i-1}(r_\tau^{w_j}+\cdots+r^{w_{\ell-1}}_\tau)=g_\tau^{x_j}+g_\tau^{x_\ell}$.
\end{lemma}

\begin{proof}
    This can be seen combinatorially.
The free resolution of $U_\tau$ corresponds to the path graph $\mathcal{P}_\tau$, in the sense that the restriction of $p^1_{i-1}$ to a map $R_\tau \to G_\tau$ corresponds to the boundary map of the path graph.
The relations $r_{\tau}^{w_j},\ldots,r_{\tau}^{w_{\ell-1}}$ then define the unique path connecting the vertices $g_\tau^{x_j}$ and $g_\tau^{x_\ell}$ in $\mathcal{P}_{\tau}$.
\end{proof}

\begin{figure}
\centering
  \includegraphics[scale=0.45]{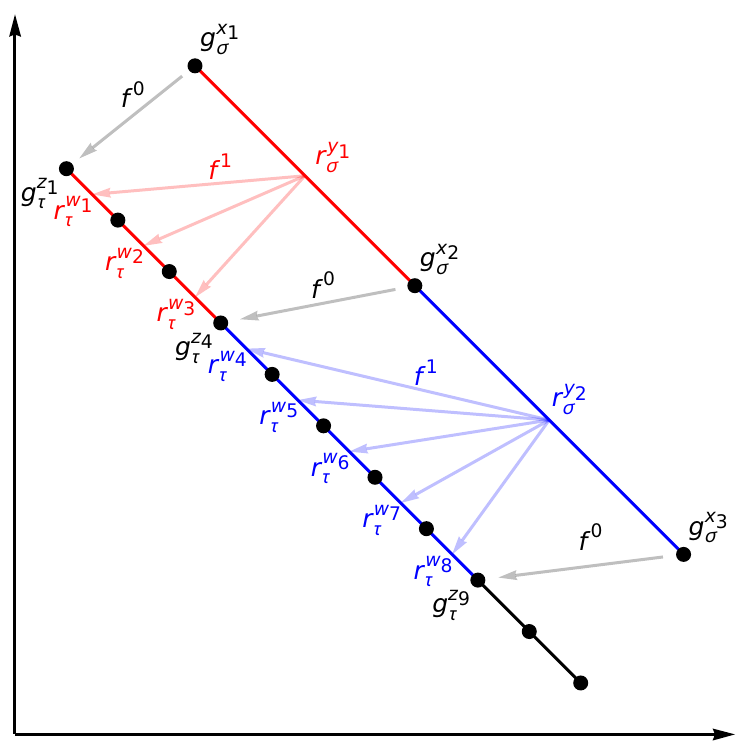}
  \caption{$f_i^1$ sends each edge $r_\sigma^y$ to the path connecting the images of its endpoints under $f_i^0$.}  
  \label{fig:f1}
\end{figure}

The map $p^1_i$ sends the edge $r_\sigma^{y_j}$ of $\mathcal{P}_\sigma$ to the vertices $g_\sigma^{x_j}$ and $g_\sigma^{x_{j+1}}$ which in turn are mapped to $g_{\tau_j}^{z_j}$ and $g_{\tau_j}^{z'_j}$ on $\mathcal{P}_{\tau_j}$ by $f^0_i$. We can therefore modify the boundary map in \eqref{diag:path_step1} by additionally sending $r_\sigma^{y_j}$ to the path $r_{\tau_j}^{w_1}+\cdots+r_{\tau_j}^{w_\ell}$ guaranteed by Lemma \ref{lem:path-lemma}, see Figure \ref{fig:f1}. This defines a correction term $f^1_i:R_i\rightarrow R_{i-1}$ such that $p^1_{i-1}\circ f^1_i = f^0_i\circ p^1_i$. The maps $f^1_i$ are called the \define{first lifts} of $\partial_i$.

Even after adding the correction term $f^1_i$ to \eqref{diag:path_step1}, the resulting maps may still fail to be boundary maps, since the composition $f^0_{i-1}\circ f^0_i(g_\sigma^x)$ need not vanish for each generator $g_\sigma^x$ in $G_i$, as seen in Figure \ref{fig:example_diagram}.
Again, $f^0_{i-1}\circ f^0_i(g_\sigma^x)$ is a sum consisting of pairs $g^{y}_\tau,g^{y'}_\tau$ of possibly different generators  of the same simplex $\tau$,  which can be connected by a path of relations as above.
A second correction term $h^0_i\colon G_i\rightarrow R_{i-2}$ is defined analogously to the lifts $f^1_i$.
By construction, $h^0_i$ satisfies $p^1_{i-2}\circ h^0_i = f^0_{i-1}\circ f^0_i$, making the following diagram commute:
\vspace{-25pt}
\begin{equation}
\label{diag:path_step3}
\adjustbox{max width=0.9\textwidth}{
\begin{tikzcd}[ampersand replacement=\&]
\& \& \phantom{a} \arrow[ddr,shorten <=30pt,"h^0_3"{xshift=7pt,yshift=-10pt},red] \\
\& \vdots \arrow[d,swap,"\partial_3"] 
  \& \vdots  \arrow[d,swap,"f_3^0"] 
  \& \vdots \arrow[d,"f_3^1",blue] \\
0 \& C_2 \arrow[l] \arrow[d,swap,"\partial_2"] 
  \& G_2 \arrow[d,swap,"f_2^0"] \arrow[l,swap,"\alpha"] \arrow[ddr,"h^0_2"{xshift=-10pt,yshift=10pt},red]
  \& R_2 \arrow[l,swap,"p^1_2"] \arrow[d,"f_2^1",blue] \& 0 \arrow[l] \\
0 \& C_1 \arrow[l] \arrow[d,swap,"\partial_1"] 
  \& G_1 \arrow[d,swap,"f_1^0"] \arrow[l,swap,"\alpha"] 
  \& R_1 \arrow[l,swap,crossing over,"p^1_1"{xshift=3pt}] \arrow[d,"f_1^1",blue] \& 0 \arrow[l] \arrow[dddrr,bend left,violet,shorten=40pt]  \\
0 \& C_0 \arrow[l] \arrow[d]
  \& G_0 \arrow[l,swap,"\alpha"] \arrow[d]
  \& R_0 \arrow[d] \arrow[l,swap,"p^1_0"] \& 0 \arrow[l] \\
  \& 0 \& 0 \& 0 \\
  \&\&\& 0 \& G_0  \arrow[l]
  \& G_1\oplus R_0 \arrow{l}[swap]{\begin{pmatrix}f_1^0 & p^1_0 \end{pmatrix}} 
  \& G_2\oplus R_1 \arrow{l}[swap]{\begin{pmatrix}f_2^0 & p^1_1 \\ h^0_2 & f^1_1\end{pmatrix}} 
  \& \cdots \arrow{l}[swap]{\begin{pmatrix}f_3^0 & p^1_2 \\ h^0_3 & f^1_2\end{pmatrix}} 
\end{tikzcd}}
\end{equation}
Adding $h^0_i$ and $f^1_i$ to the boundary maps in \eqref{diag:path_step1} yields the desired free resolution of $C_\bullet$. The following Theorem, proven in Appendix \ref{app:path_correctness}, concludes the correctness of the path algorithm.
\begin{theorem}
\label{thm:path_correctness}
The chain complex on the right side of \eqref{diag:path_step3} is a free resolution of $C_\bullet$.  
\end{theorem}

\begin{figure}
\centering
  \includegraphics[width=\textwidth]{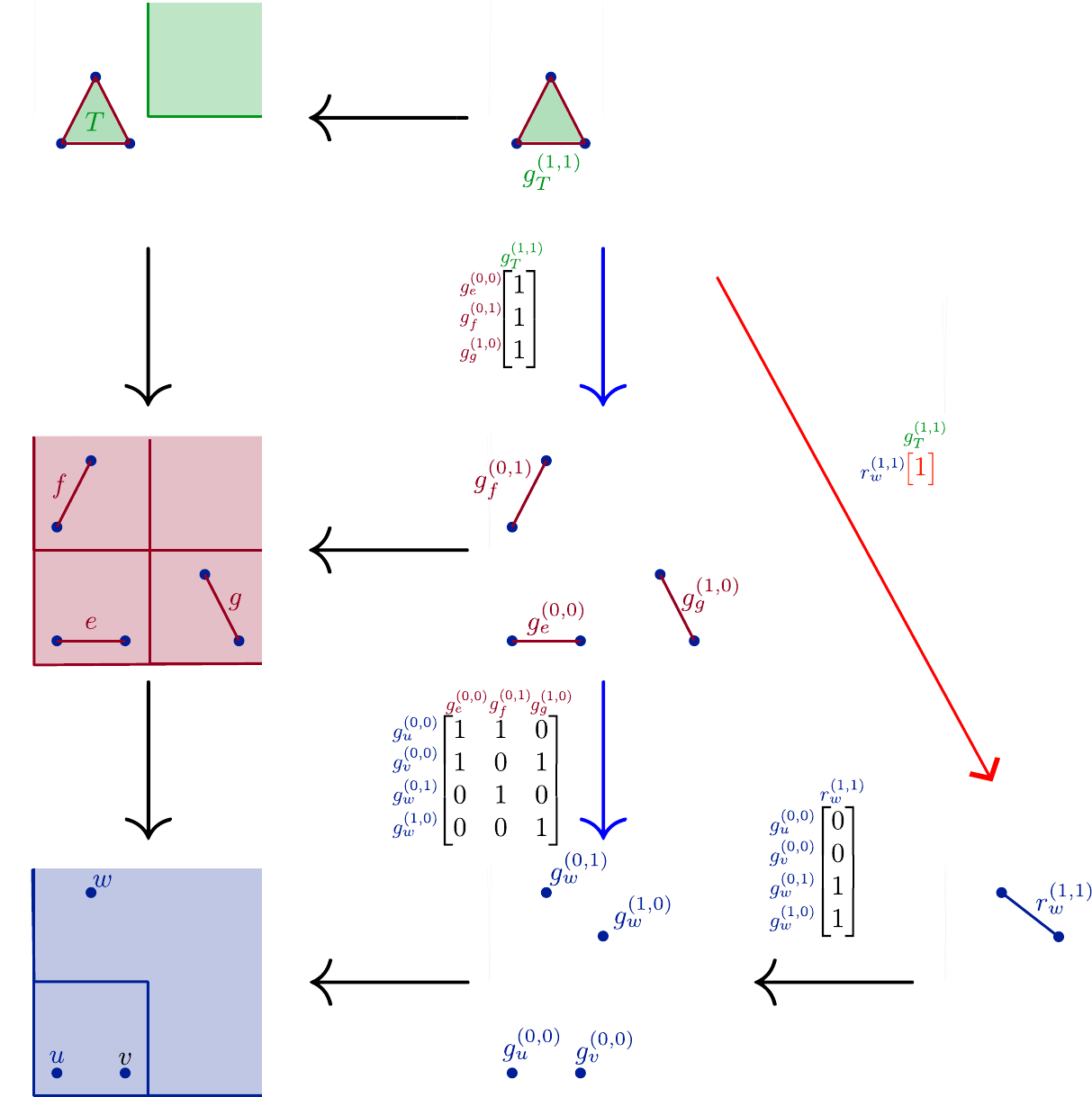}

  \vspace{1cm}

    \includegraphics[width=\textwidth]{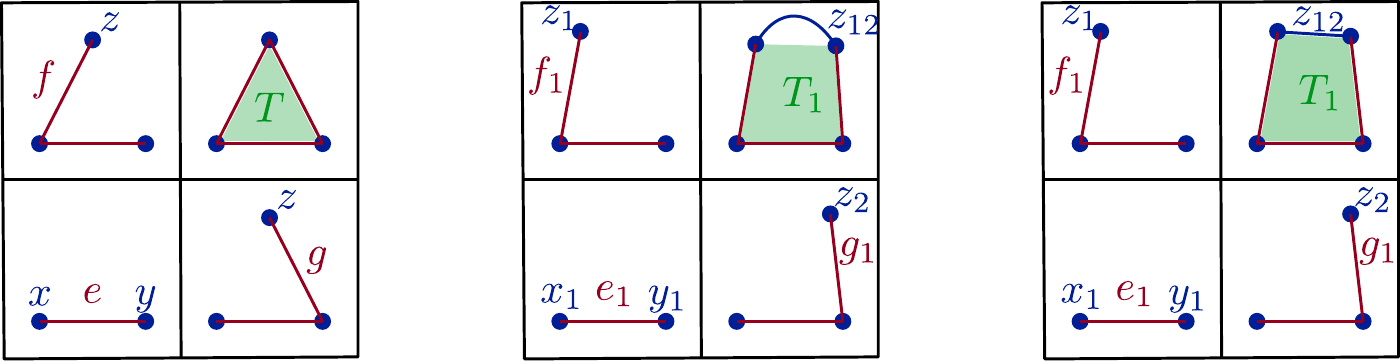}

  \caption{Construction of Diagram \ref{diag:path_step3} for the bifiltration on the bottom left. The boundary defect $f^1_{0}\circ f^1_1\neq 0$, can be repaired by adding the red correction morphism which corresponds to a change of the boundary of $T_1$ in the bottom-middle cell complex. This yields the bottom-right output.}  
  \label{fig:example_diagram}
\end{figure}

\subparagraph{Path algorithm.} For the description of the algorithm and its asymptotic time complexity, we assume that the input $C_\bullet$ is induced by a $k$-critical bifiltration of a simplicial complex with fixed dimension $d$, given as a list of simplices $\sigma$, each with its dimension, grades $\mathcal{G}(\sigma)$ and boundary $\partial\sigma$. The output is described by a list of generators and relations (each carrying a grade) as well as their images under the maps $p^1_i,f^0_i,f^1_i$ and $h^0_i$ (which together form the boundary maps of the free resolution). The output is initialized as an empty list. 
\begin{enumerate} 
\item Compute the bases of $G_i \oplus R_{i-1}$ and the matrices $[p_i^1]$.  For each $i$-simplex $\sigma$ and $x_j\in \mathcal{G}(\sigma)$, add a generator $g_\sigma^{x_j}$ to $G_i$. For any two consecutive $x_j,x_{j+1}\in \mathcal{G}(\sigma)$, add a relation $r_\sigma^{y_j}$ with $y_j=x_j\vee x_{j+1}$ to $R_{i}$. Moreover, define $p_{i}^1(r_\sigma^{y_j})\coloneqq g_\sigma^{x_j}+g_\sigma^{x_{j+1}}$. 
\item Compute the matrices $[f_i^0]$. For each generator $g_\sigma^x$ in $G_i$, find generators $g_{\tau_0}^{x_0},\ldots, g_{\tau_i}^{x_i}$ for $\tau_0+\cdots+\tau_i=\partial\sigma$ such that $x_j\leq x$ and define $f^0_i(g_\sigma^x)\coloneqq g_{\tau_0}^{x_0}+\cdots+g_{\tau_i}^{x_i}$. 
\item Compute $[f_i^1]$ and $[h^0_i]$. For $[f_i^1]$,  perform a matrix multiplication $[f_i^0][p_i^1]$. The resulting columns are indexed by the relations $r_\sigma^y$ in $R_i$ and consist of one pair of generators $g_\tau^{z_j},g_\tau^{z_\ell}$ for each facet $\tau$ of $\sigma$. 
Such a pair is connected by a path $c_\tau:=r_\tau^{w_j}+\cdots + r_\tau^{w_\ell}$ in $\mathcal{P}_\tau$, and we set $f^1_i(r_\sigma^y):=\sum_{\tau} c_\tau$; see Figure \ref{fig:first_lifts_new}.
For $[h^0_i]$, proceed analogously to multiply $[f^0_{i-1}][f^0_i]$. 
The resulting columns are indexed by the generators $g_\sigma^x$ in $G_i$ and consist of one pair of generators $g_\tau^{z_j},g_\tau^{z_\ell}$ in $G_{i-2}$ for each codimension $2$ face $\tau$ of $\sigma$. 
Such a pair is connected by a path $c_\tau:=r_\tau^{w_j}+\cdots + r_\tau^{w_\ell}$ in $\mathcal{P}_\tau$, and we set $h^0_i(g_\sigma^x):=\sum_{\tau} c_\tau$.
\end{enumerate}
\begin{figure}
  \centering
  \includegraphics[width=14cm]{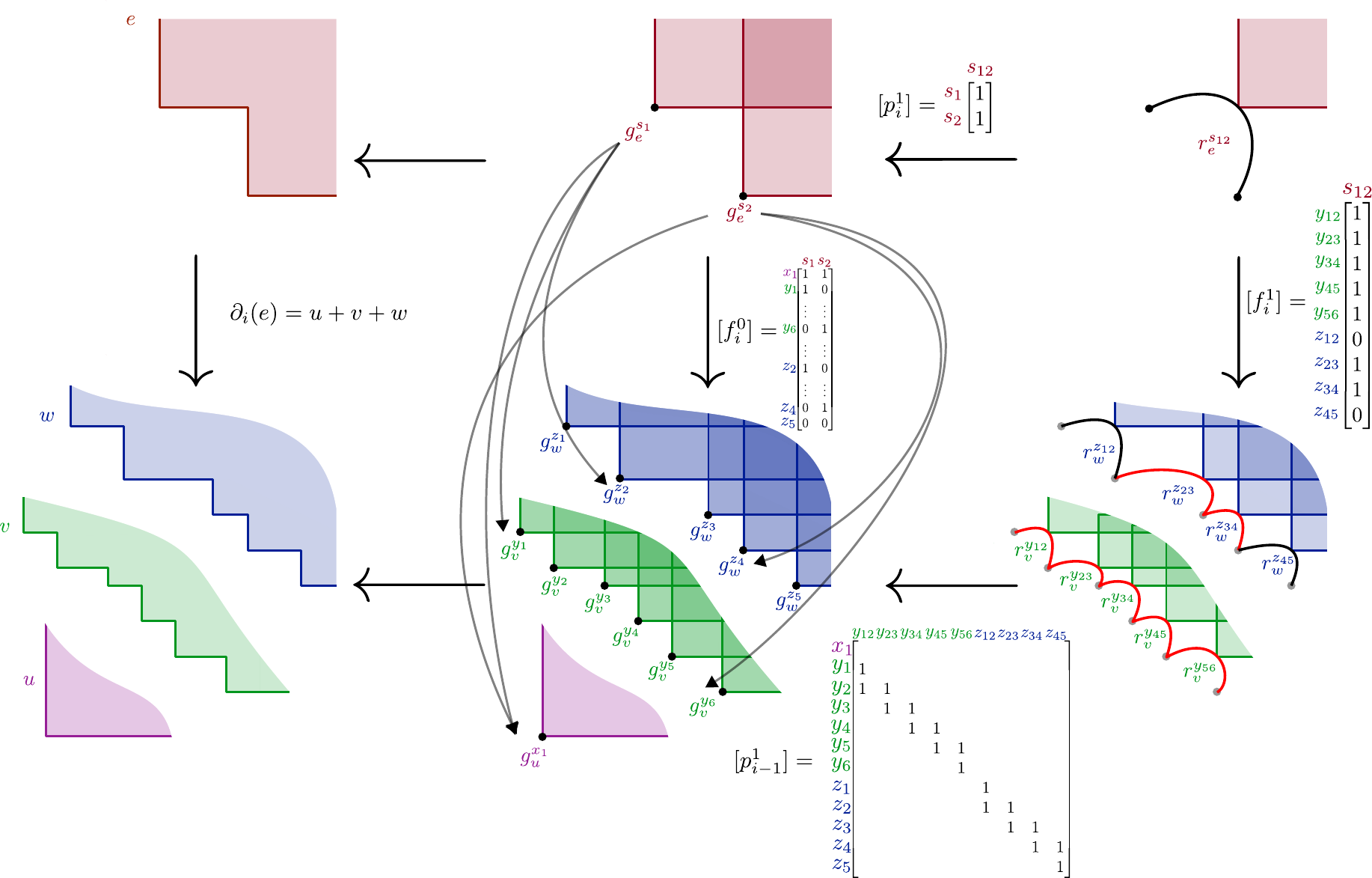}
  \caption{Construction of $f^1_i$. The relation $r_e^{s_{12}}$ is mapped to the paths $r_w^{z_{23}}+r_w^{z_{34}}$ and $r_v^{y_{12}}
  +\cdots+ r_v^{y_{56}}$ in the path resolutions of $w$ and $v$, highlighted in red.
  }
   \label{fig:first_lifts_new}
\end{figure}

\subparagraph{Complexity.}

The running time of the path algorithm is $O(nk)$ for a $k$-critical bifiltration of description size $n$; see Appendix~\ref{app:path_complexity}. 
This runtime is 
worst-case optimal, since the description size of a computed free resolution can be $\Omega(nk)$ in the worst-case.
The reason is that the paths in the computation of $f^1_i$ and $h^0_i$ can be of size up to $k$, and many simplices may require such long paths in the worst case.
Figure~\ref{fig:wheels} gives a construction for which the matrix $[h^0_2]$ is of size $\Omega(nk)$; see Figure~\ref{fig:star} in the Appendix
for a similar construction for $[f^1_1]$.

A free chain complex is \emph{minimal} if there is no quasi-isomorphic
free chain complex with a smaller number of generators. In general,
free resolutions computed by the path algorithm are not minimal.
However, a slight modification of the construction in Figure~\ref{fig:wheels}
yields the following result (see Appendix~\ref{app:wheel_example} for details).

\begin{proposition}
\label{prop:path_example}
  There is a $k$-critical simplicial bifiltration of description size $n$
  for which the path algorithm computes a free resolution represented by matrices of description size $\Omega(nk)$.
  This free resolution is minimal, and no other choice of basis yields a smaller description size.
\end{proposition}

\begin{figure}
  \centering
  \includegraphics[width=13.5cm]{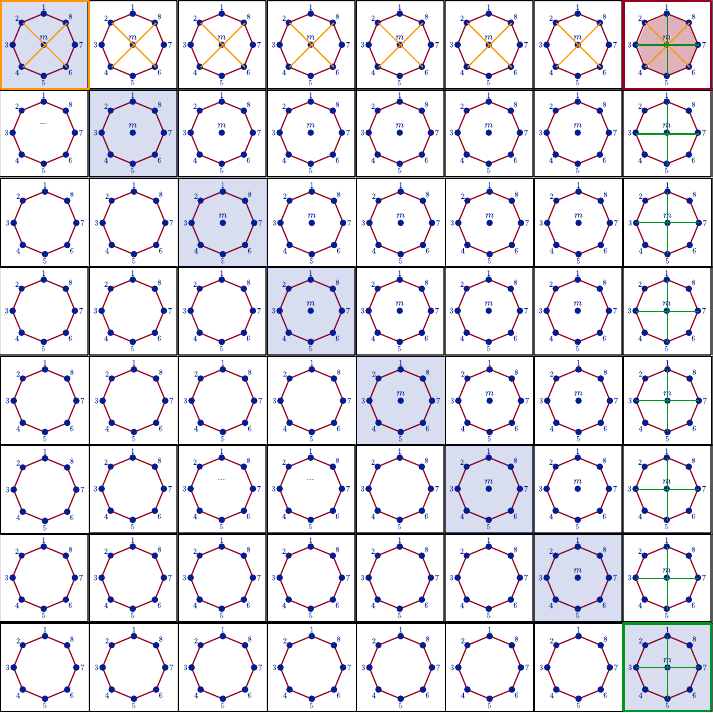}
  \caption{The \define{wheel} bifiltration consists of $\ell+1$ $0$-simplices, $\ell$ $1$-simplices and $\ell$ $2$-simplices, with $\ell=8$ in this case. The only $\ell$-critical $0$-simplex is $m$ entering the bifiltration along the blue staircase. This forces $f^0_1$ to map all generators of the orange (resp. green) $1$-simplices to a generator of $m$ with orange (resp.\@ green) grade. The composition $f^0_1 \circ f^0_2$ then maps each generator of a purple $2$-cell to the sum of a generator of $m$ with orange grade and one with a green grade. The orange and green generators of $m$ are connected by $\ell-1$ relations. This yields $\ell-1$ non-zero entries in each of the $\ell$ columns of $[h^0_i]$, making it dense.}
   \label{fig:wheels}
\end{figure}

\section{Log-path algorithm} 
\label{sec:tree}

\subparagraph{Shortcuts.} In the preceding section, we leveraged the fact that the upset modules induced by simplices of a bifiltration admit path-shaped resolutions, allowing us to compute the correction terms $f^1_i$ and $h_i^0$ in a simple way.
In the worst case, many generators and relations are mapped to long paths, making the matrices $[f^1_i]$ and $[h_i^0]$ dense ($\Omega(nk)$ size) and leading to a quadratic running time of $\Theta(nk)$ for the path algorithm. 

Such long paths can be avoided by introducing shortcuts in the path resolution $\mathcal{P}_\sigma$ \eqref{diag:free_res}, as illustrated in Figure \ref{fig:shortcut_path}.
Recall that the grades $\mathcal{G}(\sigma)=\{x_1,\cdots,x_{n_\sigma}\}$ generating $\mathrm{supp}(\sigma)$ are totally ordered (by their first coordinate).
Now any two vertices $g_\sigma^{x_j}$ and $g_\sigma^{x_\ell}$ with $j<\ell$ are connected by an additional shortcut edge if there are numbers $s \geq 0, t > 0$ with $j=s2^t$ and $\ell=(s+1)2^t$.
Extending $\mathcal{P}_\sigma$ by this edge corresponds to adding a relation $r^{z}_\sigma$ to $R_\sigma$, where $p_\sigma(r^z_\sigma)=g_\sigma^{x_j}+g_\sigma^{x_\ell}$ and $z=x_j\vee x_\ell$. In total, only $O(n_\sigma)$ many edges are added.
 
This construction ensures that any two vertices $g^{x_j}_\sigma$ and $g_\sigma^{x_\ell}$ with $j<\ell$ are connected by a monotone path of length logarithmic in $n_\sigma$.
A shortest monotone path can be constructed in a greedy way.
We start in $g_\sigma^{x_j}$ and take the longest possible edge in each step that does not overshoot $g_\sigma^{x_\ell}$.
We conclude the following extension of Lemma \ref{lem:path-lemma} (proved in Appendix \ref{app:shortest_paths}).

\begin{figure}
\centering
\includegraphics[width=14cm]{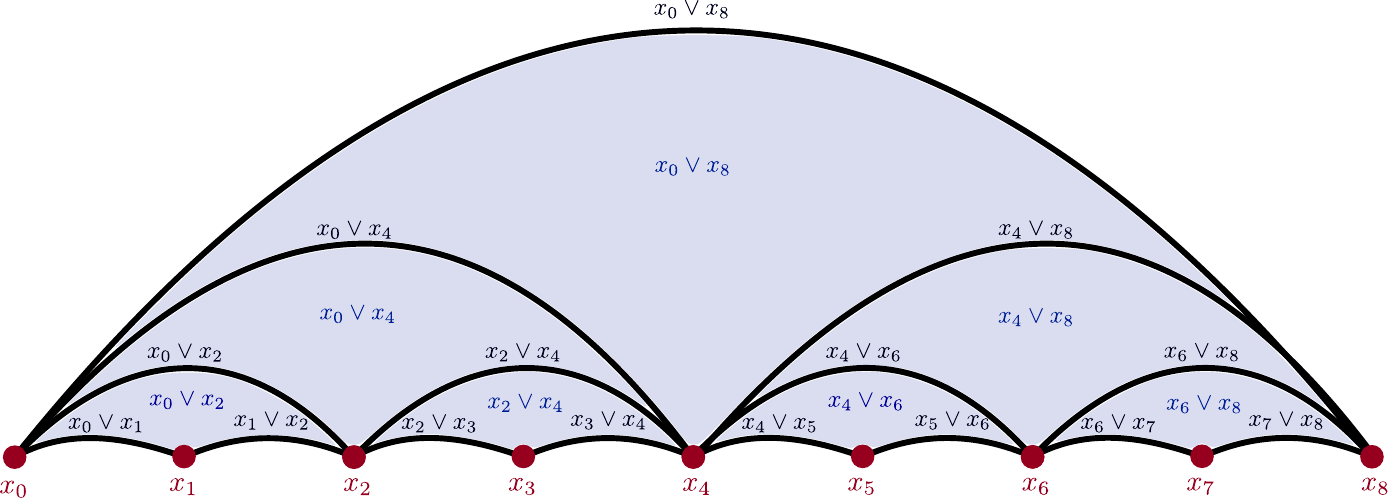}
\caption{Log-path resolution labeled by the grades of its generators, relations and syzygies. }
\label{fig:shortcut_path}
\end{figure}

\begin{figure}
\centering
\includegraphics[width=12cm]{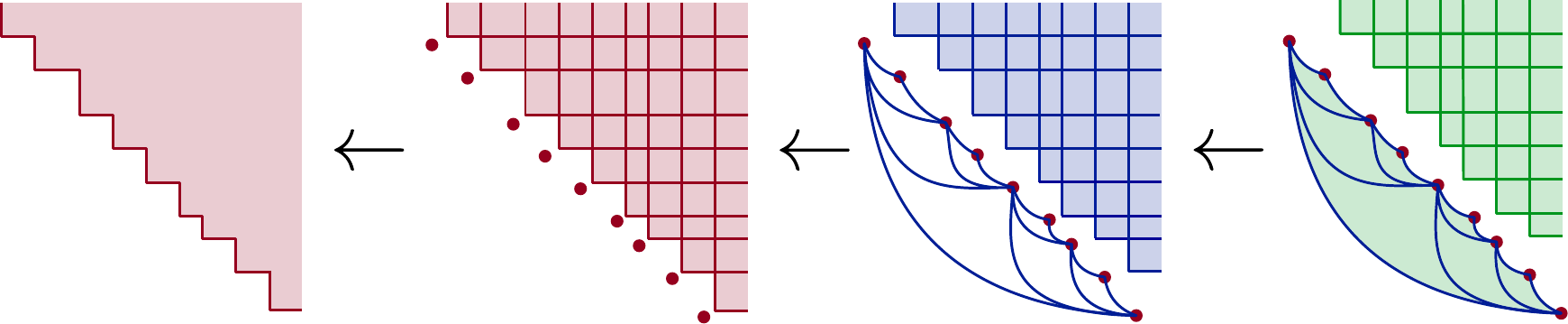}
\caption{The new resolution of length $2$ of an upset module $U_\sigma$.}
\label{fig:resolutions_of_length_two}
\end{figure}

\begin{restatable}[Log-path Lemma]{lemma}{logpathlemma}
  \label{lem:log-path-lemma}
  For any $g^{x_j}_\sigma$, $g_\sigma^{x_\ell}$ in $G_\sigma$, there exist $r_\sigma^{y_1},\ldots,r_\sigma^{y_m}$ in $R_\sigma$ with $m=O(\log n_\sigma)$
  such that $p_\sigma^1(r_\sigma^{y_1}+\cdots+r_\sigma^{y_m})=g^{x_j}_\sigma+g_\sigma^{x_\ell}$.
  Moreover, the elements $r_\sigma^{y_1},\ldots,r_\sigma^{y_m}$ can be computed in $O(\log n_\sigma)$ time.
\end{restatable}

\subparagraph{Log-path resolutions.} 
Lemma~\ref{lem:log-path-lemma} appears  to resolve the size issue: the same algorithm as in the previous section can be used, except that in the construction of $h_i^0$ and $f_i^1$ use Lemma~\ref{lem:log-path-lemma}
instead of Lemma~\ref{lem:path-lemma}.
Then, the number of nonzero boundary coefficients for every generator and relation has only size $O(\log n_\sigma)$
instead of $O(n_\sigma)$, which would lead to an output boundary matrix with $O(n\log k)$ nonzero entries. 

Now recall that a key reason for why the path construction in \eqref{diag:path_step3} yields a chain complex is the uniqueness of paths between vertices.
This leads to the vanishing of the terms $h_i^0\circ f^0_{i+1}+f^1_{i-1}\circ h_{i+1}^0$ and $h_i^0\circ p_i^1+f^1_{i-1}\circ f^1_i$ when composing boundary maps in \eqref{diag:path_step3}.
However, after adding shortcut edges the paths between vertices are no longer unique.
As illustrated in Figure \ref{fig:h1}, for a relation $r_\sigma^y$, the terms $h_i^0 \circ p_i^1(r_\sigma^y)$ and $f^1_{i-1} \circ f^1_i(r_\sigma^y)$, when restricted to a component, may correspond to different paths connecting the same endpoints and may therefore enclose a cycle. 

\begin{figure}
\centering
\includegraphics[scale=0.55]{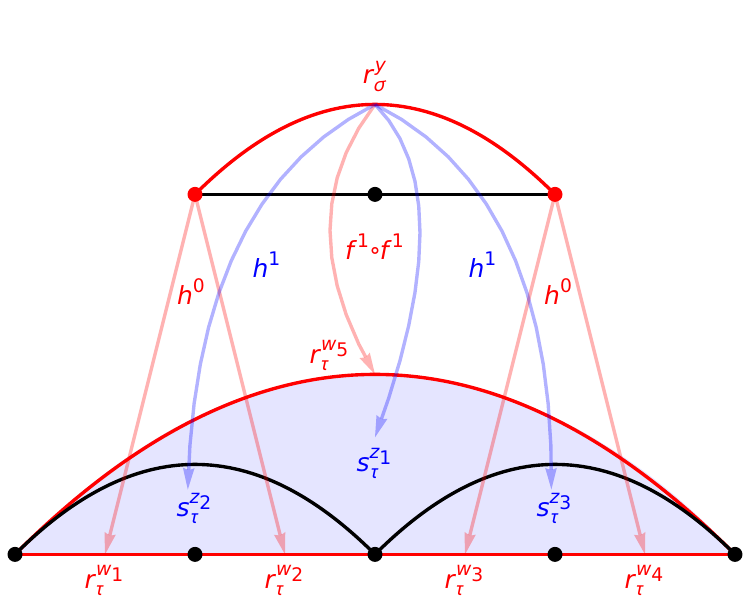}
\caption{Cycle formed by paths $h^0\circ p^1(r^y_\sigma)$ and $f^1\circ f^1(r^y_\sigma)$ and its filling triangles $h^1(r^y_\sigma)$.}
\label{fig:h1}
\end{figure}

Any such cycle constitutes an obstruction to the vanishing of $h_i^0\circ p_i^1+f^1_{i-1}\circ f^1_i$.
To eliminate these obstructions, we fill the cycles with faces, or, in algebraic terms, introduce syzygies, thus extending the modified relations to a free resolution, as depicted in Figure \ref{fig:resolutions_of_length_two}:
\begin{equation}
\label{eq:resUsigma}
\begin{tikzcd}
0 & U_\sigma \arrow[l] & G_\sigma \arrow[l,swap,"p_\sigma^0"] & R_\sigma \arrow[l,swap,"p_\sigma^1"] & S_\sigma \arrow[l,swap,"p_\sigma^2"] & 0 \arrow[l] .
\end{tikzcd}
\end{equation}
The syzygies $s_\sigma^z$ for $S_\sigma$ correspond to the triangles with vertex grades $(x_{(s-1)2^t},x_{s2^t},x_{(s+1)2^t})$ for any odd $s>0$; see Figure~\ref{fig:shortcut_path}.
The grade $z$ is given by the grade of the longest edge, that is, $x_{(s-1)2^t} \vee x_{(s+1)2^t}$.
The map $p_\sigma^2$ is defined by sending $s_\sigma^z$ to the sum of the three edges of the triangle.
We call \eqref{eq:resUsigma} the \define{log-path resolution} of $U_\sigma$, noting that it is a bifiltered two-dimensional simplicial chain complex $\mathcal{L}_\sigma$.
Its construction yields the following property:

\begin{restatable}{lemma}{cyclelemma}
  \label{ref:cycle_lemma}
  Let $r_\sigma^{y_1},\ldots,r_\sigma^{y_l}$ be a cycle in $R_\sigma$.
  Then there exists a unique chain $s_\sigma^{z_1}+\cdots+s_\sigma^{z_q}$
  of $q\leq l-2$ elements in $S_\sigma$ such that $p^2_\sigma(s_\sigma^{z_1}+\cdots+s_\sigma^{z_q})=r_\sigma^{y_1}+\cdots+r_\sigma^{y_l}$.
  Moreover, this chain can be computed
  in $O(l)=O(\log k)$ time.
\end{restatable}

To fill a cycle in $\mathcal{L}_\sigma$, we decompose it into simple cycles and use the observation that a simple cycle is filled with the triangles corresponding to the non-extremal vertices $x_{s2^t}$ of the cycle, see Figure \ref{fig:embedding}. See Appendix \ref{app:cycle_filling} for details.

\begin{figure}[t]
    \centering
    \includegraphics[scale=0.45]{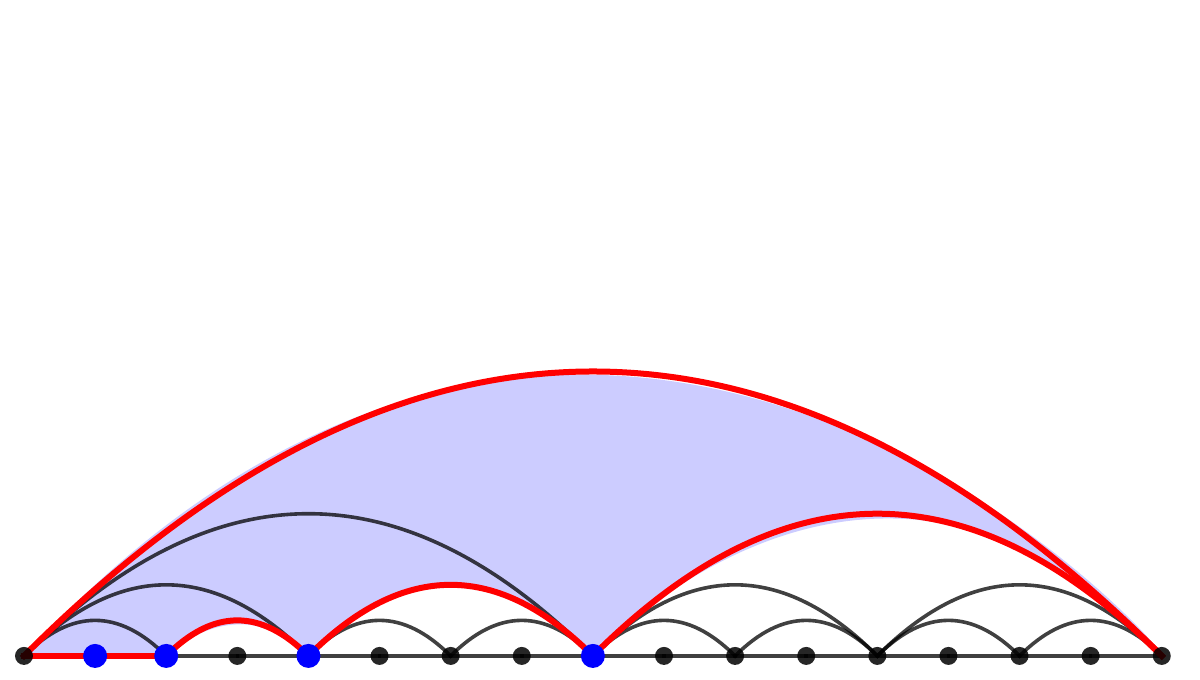}
    \caption{A simple red cycle in $\mathcal{L}$ for $k=4$, together with the filling triangles in blue corresponding to the non-extremal vertices (with respect to the total order).
    Here each triangle $(x_{(s-1)2^t},x_{s2^t},x_{(s+1)2^t})$, with $s>1$ odd, corresponds bijectively to its middle vertex $x_{s2^t}$.
    }
    \label{fig:embedding}
\end{figure}

\subparagraph{Log-path algorithm.}
We now describe the algorithm, again assuming that the input is a $k$-critical simplicial bifiltration of size $n$ and constant dimension $d$ with chain complex $C_\bullet$.

\textit{Step 1:} Compute the log-path resolution $\mathcal{L}_i:G_i\xleftarrow{p^1_i} R_i \xleftarrow{p^2_i}S_i$ of each module $C_i$. This requires iterating through the generators of each input simplex and adding a total of $O(n)$ relations and syzygies, which takes $O(n)$ time in total. 

\textit{Step 2:} Compute $f^0_i:G_i\rightarrow G_{i-1}$ in $O(n \log k)$ time as in Section \ref{sec:path}.

\textit{Step 3:} Compute the maps $f^1_i:R_i\rightarrow R_{i-1}$ and $h^0_i:G_i\rightarrow R_{i-2}$ as in the path algorithm, except that paths
between $g_\sigma^{x_j},g_\sigma^{x_\ell}$ coming from the same simplex $\sigma$
are computed via Lemma~\ref{lem:log-path-lemma}.
We can find every path in $O(\log k)$ time
and compute matrices $[f^1_i]$ and $[h^0_i]$ with at most $O(\log k)$
entries per column in a total running time of $O(n\log k)$ for this step.

\textit{Step 4.} All maps computed in the previous steps yield the following extension of \eqref{diag:path_step3}.

\vspace{-12pt}
\begin{equation} \label{diag:logpath_double_complex_incomplete}
\begin{tikzcd} [column sep=large,row sep=large]
& \vdots \arrow[d] & \vdots \arrow[d] & \vdots \arrow[d] \\
0 & C_3 \arrow[l] \arrow[d,swap,"\partial_3"] & G_3 \arrow[l,swap,"\alpha_3"] \arrow[d,swap,"f_3^0"] \arrow[ddr,"h_3^0"{xshift=-10pt,yshift=10pt}] & R_3 \arrow[l,swap,"p_3^1"] \arrow[d,crossing over,"f_3^1"] & S_3 \arrow[l,swap,"p^2_3"] & 0 \arrow[l] \\
0 & C_2 \arrow[l] \arrow[d,swap,"\partial_2"] & G_2 \arrow[l,swap,"\alpha_2"] \arrow[d,swap,"f_2^0"] \arrow[ddr,"h_2^0"{xshift=-10pt,yshift=10pt}] & R_2 \arrow[l,swap,"p_2^1"{xshift=3pt},crossing over] \arrow[d,"f_2^1"] & S_2 \arrow[l,swap,"p^2_2"{xshift=3pt},crossing over]  & 0 \arrow[l] \\
0 & C_1 \arrow[l] \arrow[d,swap,"\partial_1"] & G_1 \arrow[l,swap,"\alpha_1"] \arrow[d,swap,"f_1^0"] & R_1 \arrow[l,swap,"p_1^1"{xshift=3pt},crossing over] \arrow[d,"f_1^1"] & S_1 \arrow[l,swap,"p^2_1"{xshift=3pt},crossing over] & 0 \arrow[l] \\
0 & C_0 \arrow[l] \arrow[d] & G_0 \arrow[l,swap,"\alpha_0"] \arrow[d] & R_0 \arrow[l,swap,"p_0^1"] \arrow[d] & S_0 \arrow[l,swap,"p^2_0"] & 0 \arrow[l] \\
& 0 & 0 & 0 
\end{tikzcd}
\end{equation} 
\noindent 
A construction as in Section \ref{sec:path} leads to boundary morphisms whose composition is non-zero. Again, the boundary defect is repaired by introducing correction maps $f^2_i$, $h^1_i$, and $H^0_i$:

\vspace{-12pt}
\begin{equation} \label{diag:logpath_double_complex}
\begin{tikzcd} [column sep=large,row sep=large]
& \vdots \arrow[d] & \vdots \arrow[d] & \vdots \arrow[d] & \vdots \arrow[d] \\
0 & C_3 \arrow[l] \arrow[d,swap,"\partial_3"] & G_3 \arrow[l,swap,"\alpha_3"] \arrow[d,swap,"f_3^0"] \arrow[ddr,"h_3^0"{xshift=-10pt,yshift=10pt}] \arrow[dddrr,violet,bend left,"H_3^0"{xshift=-10pt,yshift=-14pt}] & R_3 \arrow[l,swap,"p_3^1"] \arrow[d,crossing over,"f_3^1"] \arrow[ddr,blue,"h_3^1"{xshift=-10pt,yshift=10pt}] & S_3 \arrow[l,swap,"p^2_3"] \arrow[d,red,"f_3^2"] & 0 \arrow[l] \\
0 & C_2 \arrow[l] \arrow[d,swap,"\partial_2"] & G_2 \arrow[l,swap,"\alpha_2"] \arrow[d,swap,"f_2^0"] \arrow[ddr,"h_2^0"{xshift=-10pt,yshift=10pt}] & R_2 \arrow[l,swap,"p_2^1"{xshift=3pt},crossing over] \arrow[d,"f_2^1"] \arrow[ddr,blue,"h_2^1"{xshift=-10pt,yshift=10pt}] & S_2 \arrow[l,swap,"p^2_2"{xshift=3pt},crossing over] \arrow[d,red,"f_2^2"] & 0 \arrow[l] \\
0 & C_1 \arrow[l] \arrow[d,swap,"\partial_1"] & G_1 \arrow[l,swap,"\alpha_1"] \arrow[d,swap,"f_1^0"] & R_1 \arrow[l,swap,"p_1^1"{xshift=3pt},crossing over] \arrow[d,"f_1^1"] & S_1 \arrow[l,swap,"p^2_1"{xshift=3pt},crossing over] \arrow[d,red,"f_1^2"] & 0 \arrow[l] \\
0 & C_0 \arrow[l] \arrow[d] & G_0 \arrow[l,swap,"\alpha_0"] \arrow[d] & R_0 \arrow[l,swap,"p_0^1"] \arrow[d] & S_0 \arrow[l,swap,"p^2_0"] \arrow [d]& 0 \arrow[l] \\
& 0 & 0 & 0 & 0
\end{tikzcd}
\end{equation}

Commutativity of Diagram \eqref{diag:logpath_double_complex_incomplete} yields the following observation

\begin{restatable}{lemma}{kernels}
 \label{lem:kernels}
The following morphisms map to the kernel of $p^1_{i}$:
\begin{equation}
    \label{eq:logpath_morphisms}
f^1_{i+1}\circ p^2_{i+1} , \hspace{0.5cm} f^1_{i+1}\circ f^1_{i+2} + h^0_{i+2}\circ p^1_{i+2} ,  \hspace{0.5cm} \text{and} \hspace{0.5cm} h^0_{i+2}\circ f^0_{i+3} + f^1_{i+1} \circ h^0_{i+3} .
\end{equation}   
\end{restatable}

The log-path resolution $\mathcal{L}_i$ of $C_i$ is the direct sum of the simplicial chain complexes $\mathcal{L}_\sigma$, taken over each simplex $\sigma$ in $C_i$.
The map  $p^1_i$ is given by the $1$-dimensional boundary map of this complex, and hence its kernel is the collection of cycles in the $1$-skeletons of $\mathcal{L}_\sigma$. 
More specifically, 
each triangle $s_\sigma^z$ in $S_{i+1}$ is sent by $p_{i+1}^2$ to its boundary, a sum of three edges.

By Lemma \ref{lem:kernels}, applying $f^1_{i+1}$ to this boundary yields another cycle in $R_i$, which by construction of $\mathcal{L}_i$ decomposes as $\sum_\tau c_\tau$ for $\tau$ running over the facets of $\sigma$.
For each such cycle $c_\tau$, by Lemma \ref{ref:cycle_lemma} there is a unique chain $b_\tau$ in $S_i$ such that $p^2_i(b_\tau)=c_\tau$, as illustrated in Figure \ref{fig:f2}.
The map $f^2_{i+1}$ is then defined via $f^2_{i+1}(s_\sigma^z):=\sum_{\tau}b_\tau$, with $\tau$ running over the facets of $\sigma$.
Thus $f^2_{i+1}$ repairs the boundary defect $f^1_{i+1}\circ p^2_{i+1}$ as $f^1_{i+1}\circ p^2_{i+1}=p^2_{i}\circ f^2_{i+1}$. 
\begin{figure}
\centering
\includegraphics[scale=0.55]{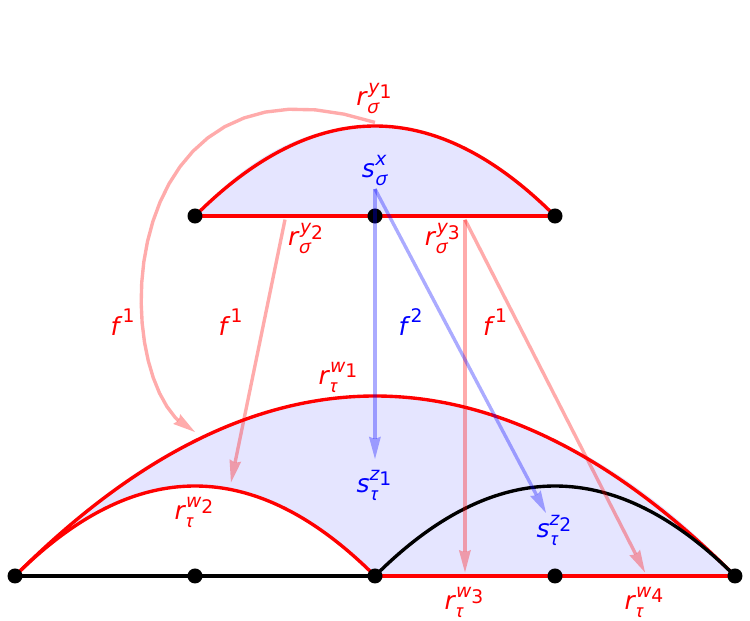}
\caption{Cycle $f^1\circ p^2(s_\sigma^x)$ and its filling triangles $f^2(s_\sigma^x)$.}
\label{fig:f2}
\end{figure}
Similar arguments lead to correction terms $h^1_i$, and $H^0_i$ satisfying 
\begin{equation}
    \label{eq:helper}
    h^0_i\circ p^1_i + f^1_{i-1}\circ f^1_i = p^2_{i-2}\circ h^1_i \hspace{0.2cm} \text{ and }\hspace{0.2cm} h^0_{i-1} \circ f^0_i + f^1_{i-2}\circ h^0_{i} =p^2_{i-3}\circ H_i^0.
\end{equation}
Their construction is analogous to the above, and we omit the details.


Note that in each of the three cases $f^2_{i+1}$, $h^1_i$, and $H^0_i$, the construction involves 
matrix multiplication (evaluation of the left hand sides in \eqref{eq:helper}), but also 
cycle filling (invoking Lemma~\ref{ref:cycle_lemma}). 
The former is efficient, as paths of logarithmic length lead to sparse matrices $[h_i^0]$ and $[f_i^1]$.
More precisely, the column sparsity of $[f_i^1]$ and $[h_i^0]$ is $O(\log k)$, while in $[p_i^2]$ every column has exactly three non-zero entries.
All sparse matrix products in Step $4$ can thus be computed in $O(n\log^2 k)$ and the worst case column sparsity of the results is $O(\log^2k)$.
The resulting matrices contain cycles of length $O(\log^2k)$, which can be filled by triangles in $O(\log^2k)$ time. Overall, Step $4$ takes $O(n\log^2k)$ time.

We can now assemble the output chain complex as
\begin{equation}  \label{diag:logpath_output}
\scalebox{0.8}{
\begin{tikzcd}[column sep=large,row sep=large,ampersand replacement=\&,every label/.append style = {font = \small}]
0 
\&[-5pt] G_0 \arrow[l]
\&[7pt] G_1\oplus R_0 \arrow{l}[swap]{\begin{pmatrix}f^0_1 & p_0^1\end{pmatrix}}
\&[12pt] G_2\oplus R_1 \oplus S_0 \arrow{l}[swap]{\begin{pmatrix}f^0_2 & p_1^1 & 0 \\ h_2^0 & f_1^1 & p_0^2\end{pmatrix}}
\&[12pt] G_3\oplus R_2\oplus S_1 \arrow{l}[swap]{\begin{pmatrix}f^0_3 & p_2^1 & 0 \\ h_3^0 & f_2^1 & p_1^2 \\ H_3^0 & h_2^1 & f_1^2 \end{pmatrix}}
\&[-5pt] \cdots \arrow[l]
\end{tikzcd}
}
\end{equation}

\begin{restatable}{theorem}{quasiisologpath}
\label{thm:quais_iso_logpath}
The chain complex \eqref{diag:logpath_output} is a free resolution of $C_\bullet$.    
\end{restatable}
For a proof, see Appendix \ref{app:quasiiso_logpath} and \ref{app:log_path_algo}.
We summarize our observations on the running time in the following theorem.

\begin{restatable}{theorem}{logpathalgo}
    \label{thm:logpath_theorem}
    A free resolution of the chain complex $C_\bullet$ induced by a $k$-critical bifiltration of description size $n$ can be computed in $O(n\log^2 k)$ time and has $O(n\log^2 k)$ size.    
\end{restatable}

\ignore{
\begin{theorem}
    \label{thm:logpath_theorem}
    A free resolution of the chain complex $C_\bullet$ induced by a $k$-critical bifiltration of description size $n$ can be computed in $O(n\log^2 k)$ time and has $O(n \log^2 k)$ size.
\end{theorem}
}

Applying this result to the modified wheel example (Appendix \ref{app:wheel_example}) yields a non-minimal free resolution with smaller description size than any minimal one. Thus, perhaps surprisingly, minimal free resolutions are not asymptotically optimal in terms 
of description size.

\section{Experimental evaluation}
\label{sec:experiments}
We implemented the path and log-path algorithm for computing a free resolution
of a bifiltered chain complex 
in a C++ library called \textsc{multi-critical} \footnote{\url{https://bitbucket.org/mkerber/multi_critical/src/main/}}. The benchmark files are available upon request.  
The code expects a bifiltration given in
scc2020 format \cite{scc2020}%
, with the difference that for every simplex an arbitrarily long
sequence of bigrades can be specified.
The output is a chain complex in ``proper''
scc2020 format, representing a free filtered chain complex.
Instead
of outputting the result of our algorithms directly, the software can also
post-process the filtered chain complex with the \emph{multi-chunk} \cite{multi_chunk}
library for minimizing the chain complex~\cite{fk-chunk,fkr-compression}. 

All experiments were performed on a workstation with an Intel(R) Xeon(R) CPU E5-1650 v3 CPU (6 cores, 12 threads, 3.5GHz) and 64 GB RAM, running Ubuntu 16.04.5.

\subparagraph{Test instances.}
We ran our experiments on a total of 342 test instances from different
sources.
First, we generated degree-Rips bifiltrations
for (noisy) point samples drawn from a torus
embedded in $\R^5$, from the ``swiss roll'' embedded in $\R^3$,
and from the $7$-dimensional unit cube, using the python package \textsc{tadasets}.
In all cases, we generated the complex up to $3$-simplices.
\ignore{
We also included \emph{compressed} degree-Rips bifiltrations in our experiments
which are currently researched in a parallel project~\cite{anonymous}.
The idea of the approach is that generators of $0$- and $1$-simplices
in the deg-Rips complex can be simplified (i.e., their criticality is reduced)
without changing the homology of the complex. This allowed us to consider instances with more
points; however, since the results were qualitatively similar to the full degree-Rips case,
we do not give details outcomes on the compressed instances in present paper.
}

Furthermore, we used a simple method for generating bifiltrations of a simplicial
complex equipped with two non-negative
real-valued functions $f_0$ and $f_1$ on its simplices,
described in Figure~\ref{fig:bifunction}.
We employed this construction on meshes from the Aim@Shape
repository\footnote{\url{http://visionair.ge.imati.cnr.it/}},
using the squared mean curvature and
the distance to the barycenter as the two functions.
Furthermore, we bifiltered the Delaunay triangulations from the same
point samples as above (torus, cube, swiss roll)
using the minimum enclosing
radius of a simplex and the
average distance
to the $\lfloor\log n\rfloor$-nearest neighbors as filter functions.
In all cases, we generated $k$-critical instances with $k=2,4,8$.
We computed the filtration values with \textsc{CGAL}~\cite{cgal-mesh,cgal-nn,cgal-meb}.

\begin{figure}
  \centering
  \includegraphics[width=6cm]{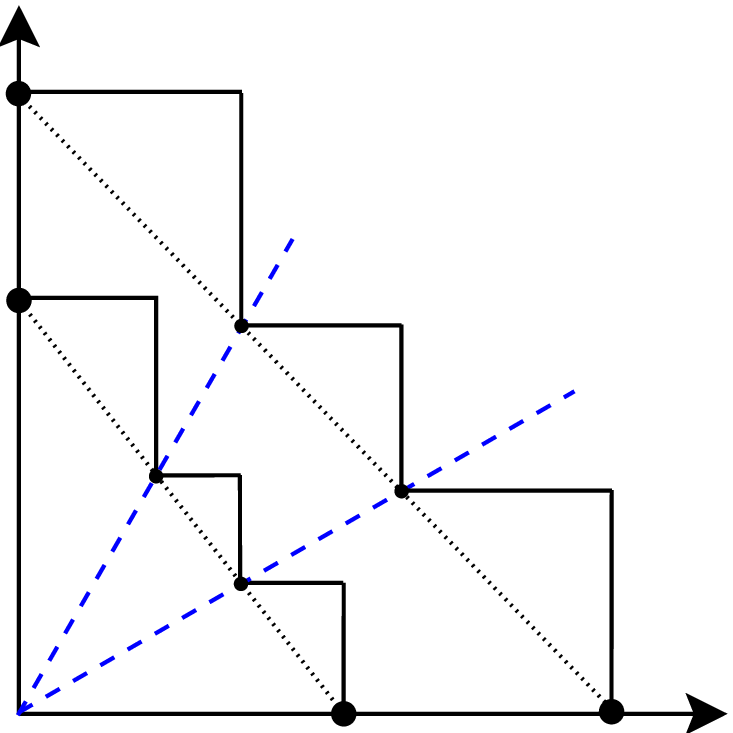}
 
  \caption{The figure shows the construction for two simplices $\sigma$ and $\tau$.
    The bigrades $(0,f_0(\sigma))$ and $(f_1(\sigma),0)$ are connected by a line segment
    (dotted line), and the positive quadrant is equally dissected using $(k-2)$ lines
    through the origin (blue dashed lines, here $k=4$). The intersection of dotted
    and dashed lines yield $(k-2)$ further bigrades for $\sigma$. The staircase
    for $\sigma$ is drawn in solid. The same construction is done for $\tau$; note that
    if $f_0(\sigma)\leq f_0(\tau)$ and $f_1(\sigma)\leq f_1(\tau)$, the staircase for $\sigma$
    is below the staircase for $\tau$. In particular, the construction bifilters
    a simplicial complex if the functions $f_0$ and $f_1$ respect its face relations.}
   \label{fig:bifunction}
\end{figure}

Finally, we generated wheel bifiltrations as in Figure~\ref{fig:wheels} for different sample sizes.

\subparagraph{Comparison of path and log-path algorithms.}
Table~\ref{tbl:path_vs_tree} shows some of the results obtained
from both algorithms on these datasets;
the results for the other instances are similar.

In almost all cases, the log-path algorithm is slower than the path algorithm
by a factor of up to 3 and produces an output complex of around twice the size.
Also, it uses
around twice as much memory (not shown in the table).
Moreover, minimizing the chain complex
(rightmost column of Table~\ref{tbl:path_vs_tree})
further reduces the complex size significantly, generally at a small cost~--
sometimes it even saves time because the time used for minimization is less than
the cost of producing the larger output file.

The only exception is the wheel bifiltration,
where our experiments show the expected asymptotic worst-case behavior, with the log-path algorithm
outperforming the path algorithm.
In this case, minimizing the
chain complex produced by the log-path algorithm will necessarily introduce a dense matrix 
and destroy the advantage of the log-path algorithm~-- for instance, its running time in
the instance in the last row of Table~\ref{tbl:path_vs_tree} is 33 seconds, with 14.4 seconds for
minimization itself and 18 seconds to write the 2.3 GB output file.

\begin{table}[h]
  \includegraphics[width=\textwidth]{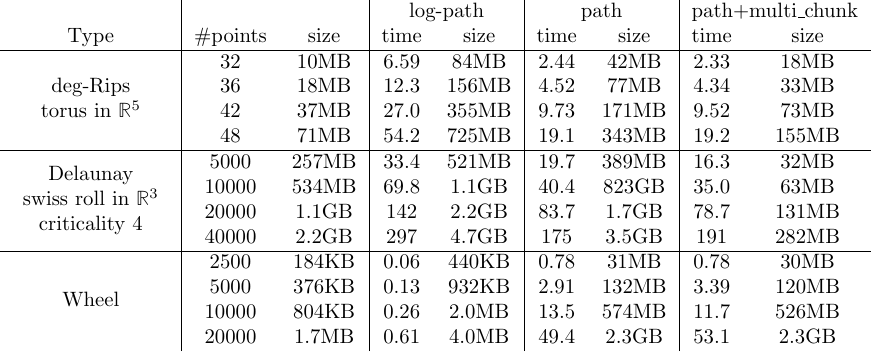}
    \caption{Comparison of the path and the log-path algorithm. All times are
      in seconds.
      The ``size'' columns give the size of the input and output files.
      For deg-Rips and Delaunay, the numbers are averaged over 5 independently generated instances.
      }
    \label{tbl:path_vs_tree}
\end{table}

\subparagraph{Computing minimal presentations.}
For a multi-critical filtered chain complex, we investigate the task of computing minimal presentation
matrices for homology.
We provide two variants:
In the first variant, we compute a free resolution of the input chain complex with the path
algorithm (which is usually faster than the log-path algorithm, as seen in the previous experiment) and subsequently
minimize using \textsc{multi-chunk}.
We then compute minimal presentations for each homology degree separately using the \textsc{mpfree} library~\cite{kr-fast,fkr-compression}.

The second variant operates by directly structuring the input complex into chain complex segments per degree.
It then uses the Chacholski--Scolamiero--Vaccarino algorithm to compute a free implicit representation for each segment, before finally employing \textsc{mpfree} once more to generate a minimal presentation.

The results of this experiment are presented in Table~\ref{tbl:compare_csv}.
For the degree-Rips instances,
the first variant provides a modest improvement over the second.
The reason
is that the vast majority of simplices are in the top dimension $3$, so that computing the
presentation for $H_2$ is the bottleneck in the computation.
This step, however, does not differ significantly
in both approaches: most of the time is spent to determine which $3$-simplices are killing $2$-cycles,
by reducing the boundary matrix for $2$- and $3$-simplices.

For the bifunction instances, the speed-up of the first variant using free resolutions is much more pronounced.
The reason is that these instances have a more balanced distribution of simplices over different dimensions:
while for the second variant, the algorithm for $H_k$ still has to find the bounding $(k+1)$-simplices for every $k$,
the multi-chunk algorithm makes use of the clearing optimization~\cite{ck-twist} and hence avoids the reduction of large parts
of the boundary matrices. Remarkably, this technique is so effective that computing \emph{all} minimal presentations
via free resolutions is faster than computing a \emph{single} minimal presentation via the approach using the Chacholski--Scolamiero-Vaccarino algorithm,
even though restricting to a single dimension $k$ allows this approach to disregard all chains in dimensions other than $k+1$, $k$, or $k-1$.

\begin{table}[h]
  \includegraphics[width=9.5cm]{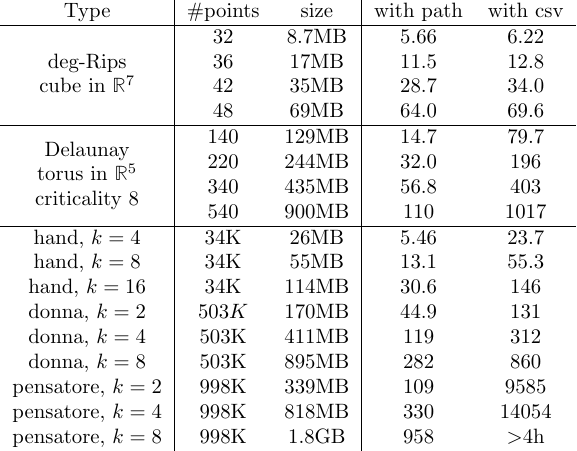}
    \caption{Computing all minimal presentations with the path
      algorithm and with the  Chacholski--Scolamiero--Vaccarino algorithm. All running times are in seconds.}
    \label{tbl:compare_csv}
\end{table}

\subparagraph{Comparison with Macaulay2.}

While the \textsc{Macaulay2} software includes a general \texttt{freeResolution} routine \cite{M2_freeres}, it is unsuitable for our purposes.
Its use requires converting our bigraded input into a graded chain complex over $\mathbb{Z}_2[x,y]$, a conversion that we found to be slow itself.
More importantly, the core computation in \textsc{Macaulay2} was orders of magnitude slower than our method, even on much smaller instances.
We infer that the software was not designed for the large inputs typical of TDA.
The conversion script is available on request.

\section{Discussion} 
\label{sec:conclusion}
Our experimental evaluation suggests that the path algorithm often exhibits slightly
better performance, which the log-path algorithm is more robust towards
``bad'' instances, with a relatively low overhead.
Together, both
variants contribute towards an efficient
computational pipeline for multi-critical bifiltrations.
Our results also complement recent development for computing degree-Rips
bifiltrations more efficiently~\cite{deg-rips-submission}.
The finding that minimal chain complexes may necessitate a quadratic size in sparse matrix representations suggests that these representations may not be universally ideal for boundary matrices.
Specifically, the matrix in Figure~\ref{fig:wheels} admits a linear-size description, illustrating a gap in current approaches.
It would be interesting to balance effective worst-case compression with efficient matrix processing, possibly by using an alternative data structure or by selective shortcutting.

Finally, our approach partially extends to simplicial complexes
filtered with three parameters: a simplexwise minimal free resolution now has length
$2$, and suitable connecting maps compose into the same diagram as 
(\ref{diag:logpath_double_complex}).
Moreover, as described by
Miller and Sturmfels~\cite{miller_sturmfels_book}, the simplexwise
free resolution carries the structure of a planar graph,
leading to a cubic-time algorithm. 
To break the cubic barrier, we will need to generalize the shortcut idea
of Section~\ref{sec:tree}
from paths to planar graphs, extending the free resolution of a simplex
to length $3$ (and introducing even more maps). We leave the details to future work.

\newpage

\bibliographystyle{plainurl}
\bibliography{socg}

\newpage

\appendix

\section{Proofs and details for the path algorithm}
\label{app:path}

\subsection{Proof of quasi-isomorphism} 
\label{app:path_correctness}

For convenience, recall that Diagram \eqref{diag:path_step3} consists of the construction 
\begin{equation}
\label{diag:app_path_double}
\begin{tikzcd}
& \vdots \arrow[d] & \vdots \arrow[d] & \vdots \arrow[d] \\
0 & C_3 \arrow[l] \arrow[d,swap,"\partial_3"] & G_3 \arrow[l,swap,"\alpha_3"] \arrow[d,swap,"f^0_3"] \arrow[ddr,dashed,"h^0_3"{xshift=-9pt,yshift=9pt}] & R_3 \arrow[l,swap,"p^1_3"] \arrow[d,"f^1_3"] & 0 \arrow[l] \\
0 & C_2 \arrow[l] \arrow[d,swap,"\partial_2"] & G_2 \arrow[l,swap,"\alpha_2"] \arrow[d,swap,"f^0_2"] \arrow[ddr,dashed,"h^0_2"{xshift=-9pt,yshift=9pt}] & R_2 \arrow[l,swap,"p_2^1"{xshift=3pt},crossing over] \arrow[d,"f^1_2"] & 0 \arrow[l] \\
0 & C_1 \arrow[l] \arrow[d,swap,"\partial_1"] & G_1 \arrow[l,swap,"\alpha_1"] \arrow[d,swap,"f^0_1"] & R_1 \arrow[l,swap,"p^1_1"{xshift=3pt},crossing over] \arrow[d,"f^1_1"] & 0 \arrow[l] \\
0 & C_0 \arrow[l] \arrow[d] & G_0 \arrow[l,swap,"\alpha_0"] \arrow[d] & R_0 \arrow[l,swap,"p^1_0"] \arrow[d] & 0 \arrow[l] \\
& 0 & 0 & 0
\end{tikzcd}
\end{equation}

whose spaces and morphisms are assembled to the output
\begin{equation}
\label{diag:app_path_output}
\scalebox{0.9}{
\begin{tikzcd}[column sep=large,row sep=large,ampersand replacement=\&,
every label/.append style = {font = \small}]
0 
  \&[-5pt] G_0 \arrow[l] 
  \&[10pt] G_1 \oplus R_0 \arrow[l,swap,"{\begin{pmatrix} f^0_1 & p^1_0 \end{pmatrix}}"] 
  \&[10pt] G_2 \oplus R_1 \arrow[l,swap,"{\begin{pmatrix} f^0_2 & p^1_1 \\ h^0_2 & f^1_1 \end{pmatrix}}"] 
  \&[10pt] G_3 \oplus R_2 \arrow[l,swap,"{\begin{pmatrix} f^0_3 & p^1_2 \\ h^0_3 & f^1_2 \end{pmatrix}}"] 
  \&[-5pt] \cdots \arrow[l]
\end{tikzcd}}
\end{equation}
which, as we will show now, is a free resolution of $C_\bullet$.

Diagram \eqref{diag:app_path_double} is commutative, that is, each square commutes and the maps $h^0_i$ satisfy
\begin{equation}
\label{eq:path_homotopy_properties}
    f^0_i \circ f^0_{i+1} = p^1_{i-1} \circ h^0_{i+1} \quad \text{and} \quad h^0_{i+1}\circ p^1_{i+1} = f^1_{i} \circ f^1_{i+1}.
\end{equation}
Here we note that the second property in \eqref{eq:path_homotopy_properties} follows from the fact that
\begin{equation*}
\begin{aligned}
p_{i-1}^1\circ (h^0_{i+1}\circ p^1_{i+1}+f^1_{i} \circ f^1_{i+1})&=f_{i}^0\circ f_{i+1}^0\circ p_{i+1}^1+f_i^0\circ p_i^1\circ f_{i+1}^1 \\ &=f_{i}^0\circ f_{i+1}^0\circ p_{i+1}^1+f_i^0\circ f_{i+1}^0\circ p_{i+1}^1=0
\end{aligned}
\end{equation*}
and $p_{i-1}^1$ is a monomorphism. Each horizontal sequence is exact, as $(G_i\xleftarrow{p^1_i} R_i,\alpha_i)$ is a free resolution of $C_i$. 

\subparagraph{Chain complex property.} We show that the sequence in Diagram \eqref{diag:app_path_output} is a chain complex. Indeed, it holds that 
\ignore{
\begin{align*}
\begin{pmatrix}f^0_i & p^1_{i-1} \\ h^0_i & f^1_{i-1} \end{pmatrix}\circ \begin{pmatrix}f^0_{i+1} & p^1_{i} \\ h^0_{i+1} & f^1_{i} \end{pmatrix}=\begin{pmatrix}f^0_{i}\circ f^0_{i+1}+p^1_{i-1}\circ h^0_{i+1} & h^0_i\circ f^0_{i+1}+f^1_{i-1}\circ h^0_{i+1} \\ f^0_i\circ p^1_i+p^1_{i-1}\circ f^1_i & h^0_i\circ p^1_i+f^1_{i-1}\circ f^1_i \end{pmatrix} = 0
\end{align*}
}
\begin{align*}
\begin{pmatrix}f^0_i & p^1_{i-1} \\ h^0_i & f^1_{i-1} \end{pmatrix}\circ \begin{pmatrix}f^0_{i+1} & p^1_{i} \\ h^0_{i+1} & f^1_{i} \end{pmatrix}=\begin{pmatrix}f^0_{i}\circ f^0_{i+1}+p^1_{i-1}\circ h^0_{i+1} & f^0_i\circ p^1_i+p^1_{i-1}\circ f^1_i \\ h^0_i\circ f^0_{i+1}+f^1_{i-1}\circ h^0_{i+1} & h^0_i\circ p^1_i+f^1_{i-1}\circ f^1_i \end{pmatrix} = 0
\end{align*}
The upper right entry is zero because $(f^0_i,f^1_i)$ is a chain map. The upper left and lower right entries are zero due to Equation \eqref{eq:path_homotopy_properties}. The postcomposition of the lower left entry with $p^1_{i-2}$ is zero by Equation \eqref{eq:path_homotopy_properties}. Thus, by exactness of the rows in Diagram \eqref{diag:app_path_double}, $ \im (h^0_i \circ f^0_{i+1}+f^1_{i-1}\circ h^0_{i+1})\subseteq \ker p^1_{i-2} = 0 $.

\subparagraph{Quasi-isomorphism.}  The upper row of the following diagram is our output chain complex. 

\begin{equation} \label{eq:quasi_iso}
\adjustbox{max width=0.92\textwidth}{
\begin{tikzcd}[column sep=large,row sep=large,ampersand replacement=\&,every label/.append style = {font = \small}]
0 \&[-5pt] G_0 \arrow[l] \arrow[d,"\alpha_0"] \&[10pt] G_1\oplus R_0 \arrow{l}[swap]{\begin{pmatrix}f^0_1 & p^1_0\end{pmatrix}} \arrow{d}{\begin{pmatrix}\alpha_1 & 0\end{pmatrix}} \&[10pt] G_2\oplus R_1 \arrow{l}[swap]{\begin{pmatrix}f^0_2 & p^1_1 \\ h^0_2 & f^1_1\end{pmatrix}} \arrow{d}{\begin{pmatrix}\alpha_2 & 0\end{pmatrix}} \&[10pt] G_3\oplus R_2 \arrow{l}[swap]{\begin{pmatrix}f^0_3 & p^1_2 \\ h^0_3 & f^1_2\end{pmatrix}} \arrow{d}{\begin{pmatrix}\alpha_3 & 0\end{pmatrix}} \&[-5pt] \cdots \arrow[l] \\[10pt]
0 \& C_0 \arrow[l] \& C_1 \arrow[l,swap,"\partial_1"] \& C_2 \arrow[l,swap,"\partial_2"] \& C_3 \arrow[l,swap,"\partial_3"] \& \cdots \arrow[l].
\end{tikzcd}}
\end{equation}

We show that $\alpha:\begin{pmatrix} \alpha_0 \end{pmatrix}, \begin{pmatrix} \alpha_1 & 0\end{pmatrix}, \begin{pmatrix} \alpha_2 & 0 \end{pmatrix},\dots$ is a quasi-isomorphism between the complex in \eqref{diag:app_path_output} and $C_\bullet$. For that reason, we show that the mapping cone $\text{cone}(\alpha)$ of the vertical maps in Diagram \eqref{eq:quasi_iso} is an acyclic complex which in turn implies that $\alpha$ is a quasi isomorphism (see Corollary $10.41$ in \cite{rotman}). $\text{cone}(\alpha)$ is the chain complex

\begin{equation} \label{eq:cone}
\scalebox{0.8}{
\begin{tikzcd}[column sep=large,ampersand replacement=\&,every label/.append style = {font = \small}]
0 \&[-12pt] C_0 \arrow[l] \&[8pt] C_1\oplus G_0 \arrow{l}[swap]{\begin{pmatrix}\partial_1 & \alpha_0\end{pmatrix}} \&[15pt] C_2\oplus G_1\oplus R_0 \arrow{l}[swap]{\begin{pmatrix}\partial_2 & \alpha_1 & 0 \\ 0 & f^0_1 & p^1_0 \end{pmatrix}} \&[15pt] C_3\oplus G_2\oplus R_1 \arrow{l}[swap]{\begin{pmatrix}\partial_3 & \alpha_2 & 0 \\ 0 & f^0_2 & p^1_1 \\ 0 & h^0_2 & f^1_1\end{pmatrix}} \&[-12pt] \cdots \arrow[l]
\end{tikzcd}}
\end{equation}
which is acyclic, if every cycle $(x,y,z)\in C_{i+1}\oplus G_i \oplus R_i$ is a boundary. Such a cycle fulfills
\begin{align}
\label{eq:boundary_path}
\begin{pmatrix}\partial_{i+1} & \alpha_{i} & 0 \\ 0 & f^0_{i} & p^1_{i-1} \\ 0 & h^0_{i} & f^1_{i-1}\end{pmatrix}\begin{pmatrix}x\\y\\z\end{pmatrix}=\begin{pmatrix}\partial_{i+1}(x)+\alpha_i(y)\\ f^0_i(y)+p^1_{i- 1}(z)\\h^0_i(y)+f^1_{i- 1}(z)\end{pmatrix}=0 .
\end{align}
\noindent
Since $\alpha_{i+1}$ is an epimorphism, there exists $a\in G_{i+1}$ such that $\alpha_{i+1}(a)=x$. To get a boundary, another summand $b\in R_i$ satisfying 
\begin{equation*}
    f^0_{i+1}(a)+p^1_i(b)=y
\end{equation*}
\noindent
is necessary. Note that $\partial_{i+1}\circ \alpha_{i+1}(a)=\alpha_i\circ f^0_{i+1} (a)$ by Diagram \ref{diag:app_path_double}. Hence by Equation \eqref{eq:boundary_path},
\begin{equation*}
    \alpha_i\circ f^0_{i+1}(a)+\alpha_i(y)=\alpha_i (f^0_{i+1}(a)+y) = 0
\end{equation*}
and thus $f^0_{i+1}(a)+y\in \ker \alpha_i = \im p^1_i$ by exactness of row $i$ in Diagram \eqref{diag:app_path_double}. Thus there exists a $b\in R_i$ such that
\begin{equation*}
    p^1_i(b)+f^0_{i+1}(a)=y.
\end{equation*}
It remains to show that 
\begin{equation}
    \label{eq:helper_pathiso}
    h^0_{i+1}(a)+f^1_i(b)=z
\end{equation}
in order to verify that $(0,a,b)\in C_{i+2}\oplus G_{i+1}\oplus R_{i}$ maps to $(x,y,z)$ by the boundary operator in \eqref{eq:cone}. Applying $p^1_{i-1}$ to Equation \eqref{eq:helper_pathiso} yields
\begin{equation*}
    p^1_{i-1}(h^0_{i+1}(a)+f^1_i(b)+z)=f^0_i(f^0_{i+1}(a)+p^1_i(b))+p^1_{i-1}(z) =                       f^0_i(y)+p^1_{i-1}(z) =0
\end{equation*}
where the first equality follows by the commutativity of Diagram \eqref{diag:app_path_double} and the third one by Equation \eqref{eq:boundary_path}. Equation \eqref{eq:helper_pathiso} now holds since $p^1_{i-1}$ is a monomorphism. This finishes the proof of Theorem \ref{thm:path_correctness}. 

\subparagraph{An algebraic remark.}

We have constructed all morphisms in Diagram \eqref{diag:app_path_double} explicitly in Section \ref{sec:path}. Their existence and properties only are guaranteed by the following Lemma \ref{lem:fundamental_lemma} from homological algebra, and only require the existence of free resolutions of length $1$ of each $C_i$.
\begin{lemma}
    \label{lem:fundamental_lemma}
    \begin{enumerate}[label=(\roman*)]
        \item For any morphism $f:M\rightarrow N$ of bipersistence modules, $F_\bullet\xrightarrow{\epsilon_M} M$ and $G_\bullet \xrightarrow{\epsilon_N} N$ free resolutions there exists a \define{lift} $L(f):F_\bullet \rightarrow G_\bullet$ such that 
        \begin{equation*}
            \begin{tikzcd}
M \arrow[d, "f"'] & F_0 \arrow[l, "\epsilon_M"'] \arrow[d, "L(f)_0"] \\
N                 & G_0 \arrow[l, "\epsilon_N"]                     
\end{tikzcd}
        \end{equation*}
        commutes. 
        \item Any two lifts $L(f)$ and $L'(f)$ are homotopic. This means that there exists a \define{chain homotopy}, which is a collection of morphisms $(h_i:F_i\rightarrow G_{i+1})_{i\in \mathbb{N}}$ such that 
        \begin{equation*}
        L(f)_i-L'(f)_i=\partial^G_{i+1}\circ h_i+h_{i-1}\circ \partial^F_i \qquad \text{for all }i\geq 0.
        \end{equation*}
    \end{enumerate} 
\end{lemma}

Indeed, the morphisms $f^0_i:G_i\rightarrow G_{i-1}$ and $f^1_i:R_i\rightarrow R_{i-1}$ are lifts of the boundary maps $\partial_i$ according to Lemma \ref{lem:fundamental_lemma} and they further assemble to a chain map $f^\bullet_i=(f^0_i,f^1_i)$ between the chain complexes $G_i\xleftarrow{p^1_i}R_i$ and $G_{i-1}\xleftarrow{p^1_{i-1}} R_{i-1}$. The composition $f^\bullet_{i-1}\circ f^\bullet_i$ is then a chain map between $G_i\xleftarrow{p^1_i}R_i$ and $G_{i-2}\xleftarrow{p^1_{i-2}}R_{i-2}$ and more specifically a lift of $\partial_{i-1}\circ\partial_i=0$. By Lemma \ref{lem:fundamental_lemma}, $f^\bullet_{i-1}\circ f^\bullet_i$ any two lifts are unique up to homotopy. As the zero map also lifts $\partial_{i-1}\circ \partial_i$, the map $f^\bullet_{i-1}\circ f^\bullet_i$ is homotopic to zero. Thus, we identify this homotopy with its only constituting map $h^0_i$. 

\subsection{Complexity}
\label{app:path_complexity}

\subparagraph{Multiplying sparse matrices.} Given two matrices $A\in\mathbb{Z}_2^{n\times p}$ and $B\in\mathbb{Z}_2^{p\times m}$ stored in sparse matrix format, i.e.\ as a list of columns represented by the row-indices of non-zero entries. Assume that the columns of $A$ and $B$ have length at most $l$ and $q$, respectively. We compute the $i$-th column of $A\cdot B=(AB_1,\ldots,AB_m)$ by summing the columns of $A$ indexed by the $i$-th column of $B$. This can be done by creating an accumulator array with zero entries of size $n$ representing one column of $A\cdot B$. We can then compute the sum of the columns of $A$ indexed by $B_i$ by accumulating the non-zero entries in this array via bit flips. This can be done by going over all columns of $A$ indexed by $B_i$ once while remembering which bits are touched. After the column $AB_i$ is finished we can clear the array. This can be done in $O(lq)$ time. Thus overall we can compute the product in $O(mql)$ time with an additional overhead of $O(n)$ for creating the accumulator array.


\begin{restatable}{proposition}{pathcorrectness}
\label{prop:path_complexity}
Given a chain complex $C_\bullet$ induced by a $k$-critical bifiltration of size $n$. The time complexity of the path algorithm is linear in the description size of its output, being $O(nk)$ in the worst case. 
\end{restatable}

\begin{proof}
The input is induced by a $k$-critical bifiltration $\mathcal{K}$ of constant dimension $d$ with description size $n$, that is, the size of the input $C_\bullet$ is the cardinality of $\bigsqcup_{\sigma\in \mathcal{K}}\mathcal{G}(\sigma)$. We assume that for each simplex $\sigma\in\mathcal{K}$, $\mathcal{G}(\sigma):=n^\sigma\leq k$. 
The size of the output is determined by all non-zero entries of the matrices $[f^0_i],[p^1_i],[f^1_i]$ and $[h^0_i]$. 

Step $1$ involves a simple iteration through all grades $\mathcal{G}(\sigma)$ and can thus be computed in $O(n)$ time. Note that all $[p^1_i]$ have size $O(n)$.

In Step $2$, we need to find generators $g_{\tau_0}^{x_0},\ldots,g_{\tau_\ell}^{x_\ell}$ for each generator $g^x_\sigma$ such that $x_j\leq x$. Here $\tau_j\in \partial\;\sigma$. Since there are at most $k$ generators for each $\tau_j$ and the dimension is at most $d$, this can be done in $O(nd\log k)$ using binary search. All $[f^0_i]$ have size $O(nd)$. 

In Step $3$, we first compute the matrix products $f^0_{i-1}\circ f^0_i$ and $f^0_i\circ p^1_i$. By construction, each column of $f_i^0$ has exactly $i+1\leq d+1$ non-zero entries, while each column of $p_i^1$ has exactly two non-zero entries. Because each involved matrix has $O(n)$ columns, these products of sparse matrices can be computed in $O(nd^2)$ time. The $O(d^2)$ non-zero entries of the columns of $f^0_{i-1}\circ f^0_i$ and $f^0_i\circ p^1_i$ consist of pairs of generators, which get connected by paths. Each such path can be found in $O(k)$ time, which results in the $[h_i^0]$ and $[f^1_{i}]$ matrices to have columns of size
$O(d^2k)$, and all $[f^1_i]$ and $[h^0_i]$ having size $O(nd^2k)$ in total. This gives an overall size and time complexity of $O(nd^2k)=O(nk)$.
\end{proof} 

\begin{figure}
  \centering
  \includegraphics[width=13.5cm]{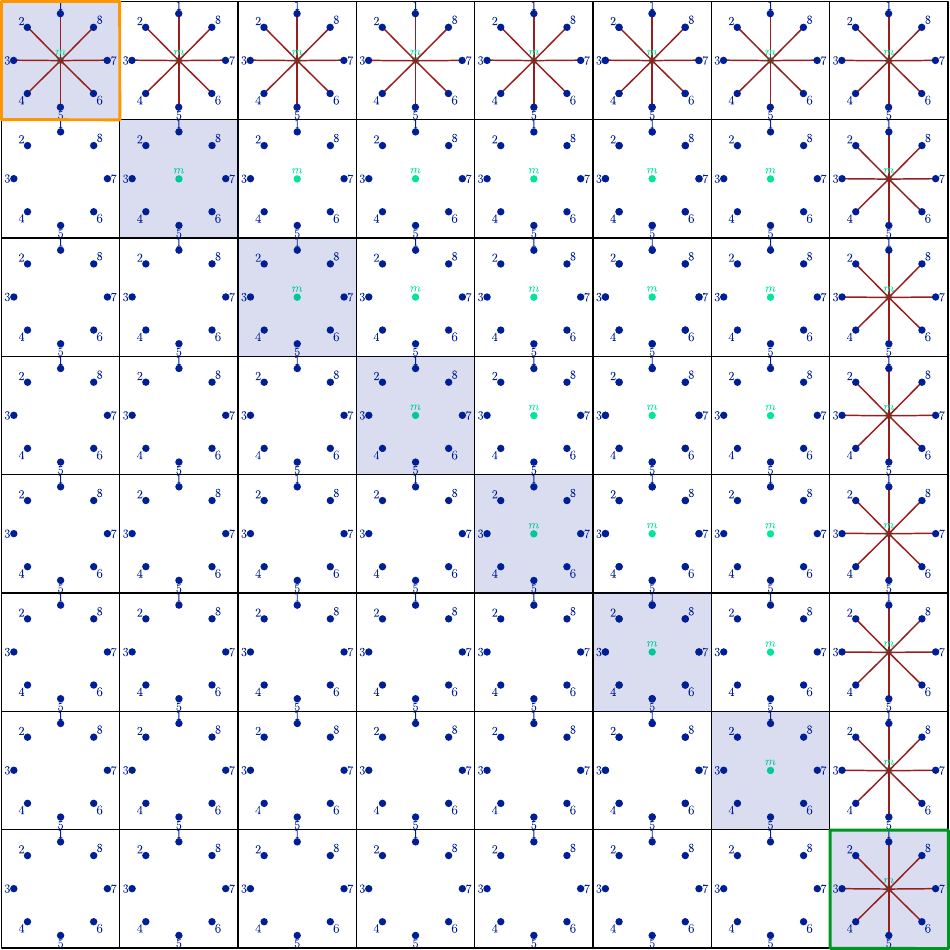}
  \caption{The \define{star}-bifiltration makes the matrix $[f^1_1]$ dense. It consists of $\ell+1$ $0$-simplices and $\ell$ $1$-simplices. The vertex $m$ is $\ell$-critical, entering the bifiltration along the blue staircase, while each $1$-simplex is $2$-critical, coming in at the orange and green grade. The map $f^0_1 \circ p_1^1$ maps each relation $r_e$ of an edge $e$ to a generator of $m$ with yellow grade and a generator of $m$ with green grade. These generators of $m$ are connected by a path of length $\ell-1$ which constitute the column of $[f^1_1]$ indexed by $r_e$.  
  }
   \label{fig:star}
\end{figure}

\subsection{Minimal resolutions of non-minimal description size} \label{app:wheel_example}

We consider the wheel example in Figure \ref{fig:wheels}. It consists of an outer cycle of $\ell$ $1$-critical vertices and edges, a central $\ell$-critical vertex $m$, $\ell$ evenly and oddly labeled $1$-critical edges that connect the vertices on the cycle to $m$ and $\ell$ $1$-critical faces. Because the only multi-critical simplex is the vertex $m$, it is the only one that induces relations. Thus, we obtain the following diagram:
\begin{equation}  \label{eq:wheel_diag}
\begin{tikzcd}
& 0 \arrow[d]  & 0 \arrow[d]  & 0 \arrow[d] \\
0 & C_2 \arrow[d,swap,"\partial_2"] \arrow[l] & G_2 \arrow[d,swap,"f^0_2"] \arrow[l] \arrow[ddr,"h_2^0"] & 0 \arrow[d] \arrow[l] & 0 \arrow[l]  \\
0 & C_1 \arrow[d,swap,"\partial_1"] \arrow[l] & G_1 \arrow[d,swap,"f^0_1"] \arrow[l] & 0 \arrow[d] \arrow[l] & 0 \arrow[l]  \\
0 & C_0 \arrow[d] \arrow[l] & G_0 \arrow[d] \arrow[l] & R_0 \arrow[d] \arrow[l,swap,"p_0^1"] & 0 \arrow[l]  \\
& 0 & 0 & 0   \\
\end{tikzcd}
\end{equation}
The map $p_0^1$ just sends each relation of $m$ to its generators. The maps $f_1^0$ and $f_2^0$ are constructed by mapping each edge and face generator to a vertex and edge generator of its boundary, respectively. The only generators where there could be choices are the generators of the edges that have $m$ in its boundary. But the edges and copies of $m$ are positioned in a way such that each edge can only be mapped to a single generator of $m$. Hence, there is no choice in the construction of these maps and the even and odd edges are mapped to the endpoints of the path formed by the generators and relations of $m$. Because each face has an even and an odd edge in its boundary and the even and odd edges are mapped to the generators corresponding to these endpoints, $f_1^0\circ f_2^0(g_\sigma^x)=g_m^y+g_m^z$ for each face $\sigma$. This implies that $h_2^0(g_\sigma^x)$ has to be defined as the path of relations of length $\ell$ connecting $g_m^y$ and $g_m^z$ for each of the $\ell-1$ faces. The output of the algorithm is the chain complex
\begin{equation} \label{eq:output_wheel}
\begin{tikzcd}[ampersand replacement=\&]
0 \& G_0  \arrow[l]
  \& G_1\oplus R_0 \arrow{l}[swap]{\begin{pmatrix}f_1^0 & p_0^1 \end{pmatrix}} 
  \& G_2 \arrow{l}[swap]{\begin{pmatrix}f_2^0 \\ h^0_2 \end{pmatrix}} 
  \& 0 \arrow[l] .
\end{tikzcd}
\end{equation}
We now modify the wheel example in Figure \ref{fig:wheels}, by slightly shifting all the edges such that every edge strictly comes after every vertex, all the edges and relations enter in incomparable grades and all the faces enter in incomparable grades without changing the relative comparability relations with the remaining simplices. This means that if two simplices (or copies thereof) are incomparable before this shift they are still incomparable after the shift. It is obvious that this can be done. After this modification every non-zero entry in the boundary matrices of \eqref{eq:output_wheel}, corresponds to a basis element $b^x$ getting mapped to a basis element $b^y$ such that $y<x$. This implies that the output chain complex is a minimal free resolution of $C_\bullet$ (see Definition 1.24 in \cite{miller_sturmfels_book}). Moreover, all basis elements corresponding to edges and relations or faces are incomparable. Therefore, the matrix $\begin{pmatrix}f_2^0 \\ h^0_2 \end{pmatrix}$ has $O(\ell^2)$ non-zero entries and there is no possible basis transformation on $G_2$ or $G_1\oplus R_0$ to reduce them. We conclude that the path-algorithm constructs a minimal free resolution from the modified wheel example that has a description size of $O(\ell^2)$. Compare this with Theorem \ref{thm:logpath_theorem} stating that it admits a free resolution of description size $O(\ell\text{ log}^2\ell)$ as produced by the log-path algorithm.  

\section{Proofs and details for the log-path algorithm} \label{app:logpath}

\subsection{Finding shortest paths} \label{app:shortest_paths}

\ignore{
\begin{figure}[t]
\adjustbox{width=\textwidth}{
\begin{tikzcd}[column sep=small]
\bullet \arrow[r,dash,shorten=-5pt] \arrow[rr,dash,shorten=-5pt,bend left] \arrow[rrrr,dash,shorten=-5pt,bend left] \arrow[rrrrrrrr,dash,shorten=-5pt,bend left] \arrow[rrrrrrrrrrrrrrrr,dash,shorten=-5pt,bend left] & \bullet \arrow[r,dash,shorten=-5pt] & \bullet \arrow[r,dash,shorten=-5pt] \arrow[rr,dash,shorten=-5pt,bend left] & \bullet \arrow[r,dash,shorten=-5pt] & \bullet \arrow[r,dash,shorten=-5pt] \arrow[rr,dash,shorten=-5pt,bend left] \arrow[rrrr,dash,shorten=-5pt,bend left] & \bullet \arrow[r,dash,shorten=-5pt] & \bullet \arrow[r,dash,shorten=-5pt] \arrow[rr,dash,shorten=-5pt,bend left] & \bullet \arrow[r,dash,shorten=-5pt] & \bullet \arrow[r,dash,shorten=-5pt] \arrow[rr,dash,shorten=-5pt,bend left] \arrow[rrrr,dash,shorten=-5pt,bend left] \arrow[rrrrrrrr,dash,shorten=-5pt,bend left] & \bullet \arrow[r,dash,shorten=-5pt] & \bullet \arrow[r,dash,shorten=-5pt] \arrow[rr,dash,shorten=-5pt,bend left] & \bullet \arrow[r,dash,shorten=-5pt] & \bullet \arrow[r,dash,shorten=-5pt] \arrow[rr,dash,shorten=-5pt,bend left] \arrow[rrrr,dash,shorten=-5pt,bend left] & \bullet \arrow[r,dash,shorten=-5pt] & \bullet \arrow[r,dash,shorten=-5pt] \arrow[rr,dash,shorten=-5pt,bend left] & \bullet \arrow[r,dash,shorten=-5pt] & \bullet \\[-20pt]
0 & 1 & 2 & 3 & 4 & 5 & 6 & 7 & 8 & 9 & 10 & 11 & 12 & 13 & 14 & 15 & 16
\end{tikzcd}}
\caption{The log-path graph for $k=4$.}
\label{fig:tree}
\end{figure}
}

\begin{figure}[t]
    \centering
    \includegraphics[scale=0.5]{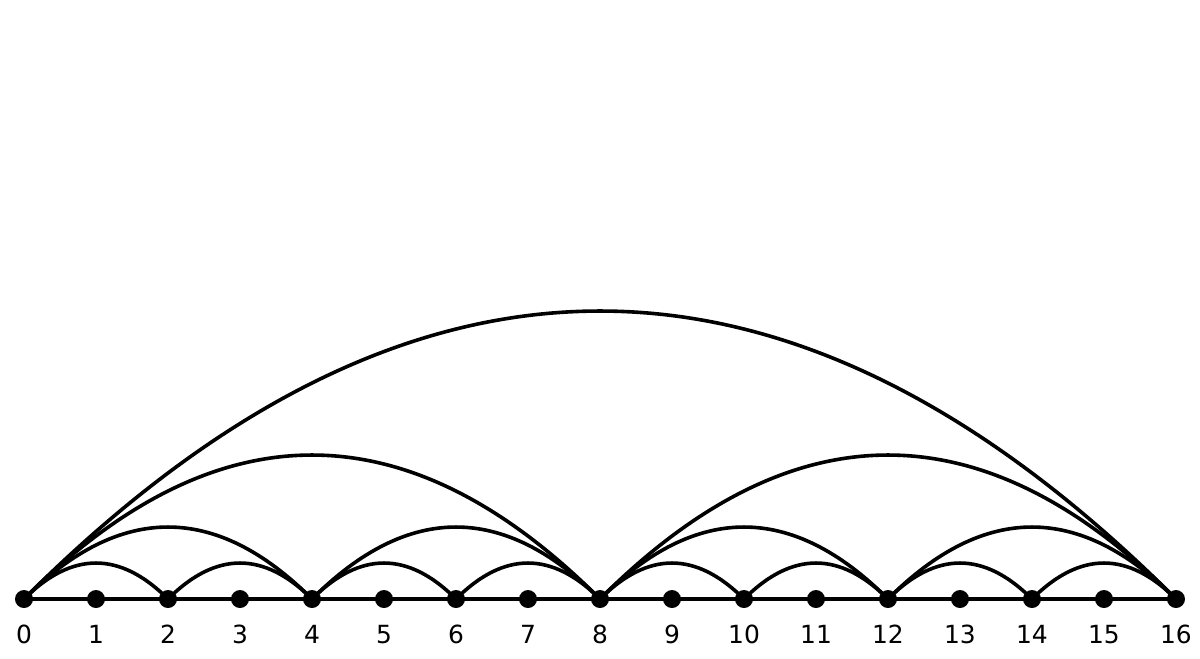}
    \caption{The log-path graph for $t=4$.}
    \label{fig:tree}
\end{figure}

In this section, we prove:
\logpathlemma*

To this end, we treat the $1$-skeleton of the log-path resolution $\mathcal{L}_\sigma$ on a purely graph-theoretic level. To do so, we identify all generators as vertices labeled by numbers and define the construction as follows: We start with a path graph of length $m=2^t$ and introduce $m-1$ additional edges, functioning as shortcuts, such that any two vertices can be connected by a monotone path of length $O\big(\text{log}(m)\big)=O(t)$.
As illustrated in Figure \ref{fig:tree}, we first add shortcuts of length two for every other vertex. Then we add shortcuts of length four at every fourth vertex, shortcuts of length eight at every eighth vertex and so on. If $m$ is not a power of two, we only add those edges that do not overshoot the last vertex of the path. Formally, for $t\in \mathbb{N}_0$, define the (undirected) graph $L=(V,E)$, where
\begin{equation*}
\begin{aligned}
V&=\{0,\ldots,2^t\} \\ 
E&=\{(x,y)\vert\exists r\in\mathbb{N}_0\colon 2^r\vert x\text{ and } y=x+2^r\}.
\end{aligned}
\end{equation*}
Note that $\vert V\vert=2^t+1$ and $\vert E\vert=\sum_{i=0}^t \frac{2^t}{2^i}=2^t \sum_{i=0}^t 2^{-i}=2^t(2-2^{-t})=2^{t+1}-1$.

Let $x,y\in L$ such that $x<y$, we call a path $x=z_0,z_1,\ldots,z_l=y$ in $L$ from $x$ to $y$ monotone if $z_i<z_{i+1}$ for all $0\leq i<l$. We can construct a shortest monotone path between $x$ and $y$ in the following way.

\textbf{Shortest monotone path algorithm:} Start at $x=z_0$ and choose a maximal $r_0$ with the property that $2^{r_0}\vert z_0$ and $z_0+2^{r_0}\leq z_l$. Set $z_1\coloneqq z_0+2^{r_0}$. In other words, take the biggest possible step towards $z_l$ that does not overshoot it. Then repeat this step by setting $z_{i+1}\coloneqq z_i+2^{r_i}$ with $r_i$ maximal such that $2^{r_i}\vert z_i$ and $z_i+2^{r_i}\leq z_l$ until $z_l=y$ is reached. 

The proof of Lemma \ref{lem:path-lemma} now follows from the following Lemma.

\begin{lemma} \label{prop:shortest_path}
Given $x<y\in L$, the algorithm above computes the unique shortest monotone path $x=z_0,\ldots,z_l=y$ of length $O(t)$ in $O(t)$ time.
\end{lemma}

The following technical lemma is not only essential for the proof of Lemma \ref{prop:shortest_path}, it also guarantees the planarity of $L$. 
Indeed, we can always draw $L$ as in Figure \ref{fig:tree}. If $(x,z)$ is an edge in $L$ and $(y,w)$ an edge such that $x<y<z$, then $x<y<w<z$. In words, all edges starting below the edge $(x,z)$ stay below $(x,z)$. In particular, $L$ is a planar graph.

\begin{lemma} \label{prop:crossing}
If $(x,z)$ is an edge in $L$ and $x<y<z\in L$ such that $2^r\vert y$, then $y+2^r\leq z$.  
\end{lemma}

\begin{proof}
If $(x,z)$ is an edge in $L$, there exists $q$ such that $2^q\vert x$ and $z=x+2^q$. Since $x<y<z$, we have $0<y-x<z-x=2^q$. If $r\geq q$, then $2^q\vert y$ and, thus, $2^q\vert(y-x)$. But this would imply $y-x=2^qa$ with $a\geq 1$ and $2^qa<2^q$, which is a contradiction. Hence, we have $r<q$ which implies $2^r\vert x$ and $2^r\vert(y-x)$. This allows us to write $y-x=2^rb$ with an integer $b\geq 1$. Since $y-x<z-x$, we get $2^rb<2^q$ and $b<2^{q-r}$ or $b+1\leq 2^{q-r}$. Therefore,
\begin{equation*}
y+2^r=x+(y-x)+2^r=x+2^rb+2^r=x+2^r(b+1)\leq x+2^r2^{q-r}=x+2^q=z.\qedhere
\end{equation*}
\end{proof}

\begin{proof}[Proof of Lemma \ref{prop:shortest_path}]
By construction, $z_1=x+2^{r_0}$, where $r_0$ is maximal with the property $2^{r_0}\vert x$ and $x+2^{r_0}\leq y$. This is the largest possible step towards $y$. Any monotone path from $x$ to $y$ that does not have $(x,z_1)$ as its first edge must use a shorter first edge $(x,a)$ with $x<a<z_1$. But by Lemma \ref{prop:crossing}, any edge starting at $a$ ends at or before $z_1$. Thus, every monotone path from $x$ to $y$ has to visit the vertex $z_1$. Since we could replace the part of any monotone path going from $x$ to $z_1$ by the single edge $(x,z_1)$ any shortest monotone path has to use $(x,z_1)$ as its first edge. By applying the same argument to the path from $z_1$ and $y$, we obtain that the construction above yields the unique monotone shortest path from $x$ to $y$.   

Let $x=z_0,\ldots,z_l=y$ be the shortest monotone path as constructed as above. Let $r_i$ be the integer such that $2^{r_i}\vert z_i$ and $z_{i+1}=z_i+2^{r_i}$. If $r_i$ is maximal with the property that $2^{r_i}\vert z_i$, then $z_i=2^{r_i}a_i$ with $a_i$ odd. Thus, $z_{i+1}=z_i+2^{r_i}=2^{r_i}a_i+2^{r_i}=2^{r_i}(a_i+1)$ with $(a_i+1)$ even and $r_{i+1}>r_i$ as long as $z_{i+1}+2^{r_{i+1}}\leq y$. Therefore, we take increasingly larger steps until we either reach $y$ or reach a point where the biggest possible step would overshoot $y$. Since $y-x\leq 2^t$, we can take at most $t-1$ of these increasing steps. If we reach a point $z_j$ where $z_j+2^{r_{j}}\leq y$ but $z_j=2^{r_j}b_j$ with $b_j$ even, then $z_{j+1}=z_j+2^{r_j}=2^{r_j}(b_j+1)$ with $(b_j+1)$ odd. If $r_{j+1}=r_j$, then $z_{j+2}=z_{j+1}+2^{r_j}=z_j+2^{r_j+1}$ but this overshoots $y$ by construction of $r_j$. Thus, $r_{j+1}<r_{j}$. By repeating this argument we have to take steps of decreasingly smaller size. Since $r_j$ is at most $t-1$, except for the case of $x=0$ and $y=2^t$ where the shortest path is of length one, we can take at most $t-1$ such decreasing steps. Therefore, overall the shortest path is of length at most $2(t-1)$. This bound is sharp as it it realized by the shortest monotone path from $1$ to $2^{t}-1$. 

The argument above also shows that the algorithm does not have to check all $t$ possible values to find the maximal $r_i$ such that $2^{r_i}\vert z_i$ and $z_i+2^{r_i}\leq y$. We only have to increase $r_i$ up to the point where we would first overshoot $y$ and from there on we only decrease it. Hence, we only have to scan through all possible values of $t$ at most twice. Since the constructed path has length smaller than $2(t-1)$, this algorithm takes at most $2(t-1)+2t$ steps.   
\end{proof}

\subsection{Filling Cycles} \label{app:cycle_filling}

\begin{figure}[t]
    \centering
    \includegraphics[scale=0.5]{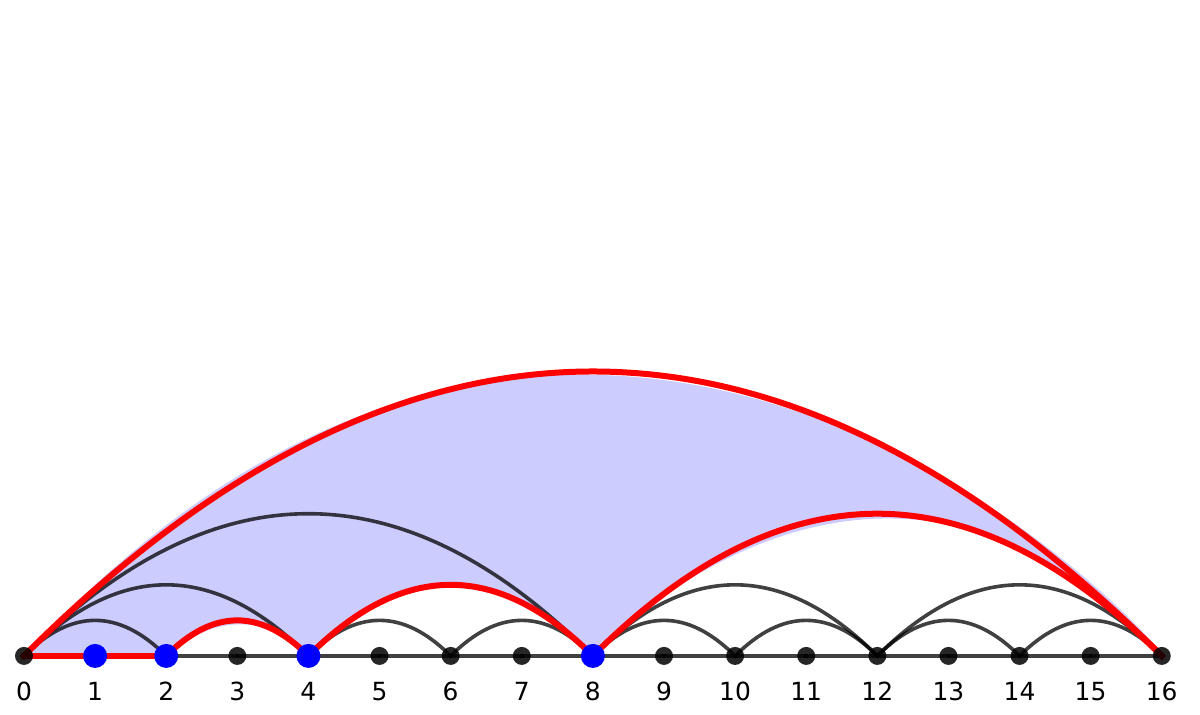}
    \caption{A simple red cycle in $L$ for $t=4$, together with the filling triangles in blue corresponding to the inner vertices.}
    \label{fig:triangulation}
\end{figure}

In this section we prove:

\cyclelemma*

The log-path resolution $\mathcal{L}_\sigma$ has the structure of $L$ viewed as a simplicial complex where all the inner triangles are filled. The triangles correspond to the syzygies of $\mathcal{L}_\sigma$. In other words, we have to solve the task of finding triangles that fill a given cycle in $L$, as illustrated in Figure \ref{fig:triangulation}. In this Section, we describe an efficient algorithm to solve this task. The Lemma then follows directly from Lemma \ref{prop:cycle_finding}.

We first focus on simple cycles, i.e., cycles that have no repeating vertices (self intersections). Moreover, we assume that the cycles are fully canceled over $\mathbb{Z}_2$, i.e., there are no repeating edges. 

First we observe that we can identify every triangle with a vertex. By construction, every triangle consists of two edges of length $2^r$ and one edge of length $2^{r+1}$. In other words, each triangle can be written as $(a,b,c)$ where $a<b<c$, $b-a=2^r$ and $c-b=2^r$. We can identify each triangle with the vertex opposite to the longest edge, i.e., $(a,b,c)\sim b$. If $(a,c)$ is the longest edge of a triangle, then $b=\frac{a+c}{2}$. The only vertices that are not matched with a triangle in this way are $0$ and $2^t$. Conversely, if $0<x<2^t$ is a vertex in $L$, and $r$ is maximal such that $2^r\vert x$, then $(x-2^r,x,x+2^r)$ is a triangle in $L$. Thus, we also identify $x\sim (x-2^r,x,x+2^r)$ and obtain a bijection between the interior vertices on $L$ and the triangles. 

\begin{lemma} \label{prop:longest_edge}
Each simple cycle in $L$ has a unique longest edge.
\end{lemma}

\begin{proof}
Let $x_0,\ldots,x_l$ be the vertices of a simple cycle in $L$. It is clear that there exists a longest edge. W.l.o.g.\ assume that $(x_0,x_l)$ is a longest edge and $x_0<x_l$. This edge needs to have length $x_l-x_0\geq 2$, as otherwise $x_0,\ldots,x_l$ would not be a cycle. Each edge of length greater than one cuts the graph into two parts, as depicted in Figure \ref{fig:triangulation}. We can visualize it as a part lying under an arc and a part over an arc.  Since the cycle is simple and the graph planar, we can not repeat a vertex and the remaining cycle has to completely lie on one of the two sides. If this cycle lies under $(x_0,x_l)$, then that same edge needs to be its longest edge by Lemma \ref{prop:crossing}. If this cycle lies over $(x_0,x_l)$, then  the only way to get back to $x_0$ from $x_l$ is to take a longer edge over $(x_0,x_l)$ which contradicts that $(x_0,x_l)$ is the longest edge. Therefore, $(x_0,x_l)$ is the unique longest edge of the simple cycle.    
\end{proof}

\begin{lemma} \label{prop:filling_simple_cycles}
If $x_0,\ldots,x_l$ are the vertices on a simple cycle in $L$ with longest edge $(x_0,x_l)$, then we can fill the cycle with the triangles corresponding to $x_1,\ldots,x_{l-1}$. In other words,
\begin{equation*}
\partial\left[\sum _{i=1}^{l-1}(x_i-2^{r_i},x_i,x_i+2^{r_i})\right]=(x_0,x_l)+\sum_{i=0}^{l-1} (x_i,x_{i+1}).
\end{equation*}
\end{lemma}

\begin{proof}
We have already observed (in the proof of Lemma \ref{prop:longest_edge}) that the path $x_0,\ldots,x_l$ has to lie under the longest edge $(x_0,x_l)$. The point $\frac{x_0+x_l}{2}$ also lies in the part of the graph lying under $(x_0,x_l)$ and further cuts it into two components (the only edge of the induced subgraph going over it is $(x_0,x_l)$). Thus, the cycle has to visit the point $\frac{x_0+x_l}{2}=x_i$ for some $0<i<l$. By adding the edges $(x_0,x_i)$ and $(x_i,x_l)$, we obtain two simple cycles $x_0,\ldots,x_i$ and $x_i,\ldots,x_l$ which might be trivial (consisting of two edges $(x_0,x_i)$ or $(x_i,x_l)$). In each of these simple cycles $(x_0,x_i)$ and $(x_i,x_l)$ are the longest edge, because they are the longest edges in the respective induced subgraphs. If $\Delta_1$ and $\Delta_2$ are collections of triangles that fill the cycles $x_0,\ldots,x_i$ and $x_i,\ldots,x_l$, respectively, then $\Delta_1+\Delta_2+(x_0,x_i,x_l)$ fills the cycle $x_0,\ldots,x_l$. Therefore, we can repeat the same argument for non-trivial cycles until they become trivial. In this way, we obtain that $\sum _{i=1}^{l-1}(x_i-2^{r_i},x_i,x_i+2^{r_i})$ fills $x_0,\ldots,x_l$.
\end{proof}

By Lemma \ref{prop:filling_simple_cycles}, we can fill a simple cycle of length $l$ in $O(l)$ time. If we are given an arbitrary fully canceled cycle $x_0,\ldots,x_l$, we can decompose it into simple cycles and then apply the same argument again. If we are only given an unordered cycle graph, we can compute an ordered closed walk in linear time by a simple greedy traversal that marks visited edges (sometimes referred to as Hierholzer's algorithm). Hence, we assume the a closed walk as an input and use the following algorithm to decompose it into simple cycles. \\

Input: a closed walk $x_0,\ldots,x_l$. Maintain a stack $S$ of vertices and a map $\text{pos}\colon V\rightarrow \mathbb{N}_0$ that stores
the index of a vertex in $S$.

\begin{enumerate}
\item Initialize $S\gets[\,]$, $\text{pos}[-]=-1$, and an empty list $\mathcal{C}$ of cycles.
\item For $i=0,1,\dots,l-1$:
  \begin{itemize}
  \item If $\text{pos}[x_i]=-1$: push $x_i$ onto $S$ and set $\text{pos}[x_i]\gets |S|-1$.
  \item Else (a repeat): let $j=\text{pos}[x_i]$. Append the cycle $(S[j], S[j{+}1], \dots, S[|S|-1])$ to $\mathcal{C}$. Then pop $S[j+1],\dots,S[|S|-1]$ off the stack and set $\text{pos}[S[r]]=-1$ for $j< r\leq\vert S\vert-1$ (keep $S[j]=x_i$).
  \end{itemize}
\item Return $\mathcal{C}$.
\end{enumerate}

The algorithm goes over the closed walk and remembers which vertices are already visited. As long as we do not hit an already visited vertex all vertices on the current stack are distinct. If the walk hits an already visited vertex the first time, then the part of the walk since the repeated vertex forms a simple cycle. After removing this simple cycle from the stack the remaining vertices form a vertex distinct walk again. We proceed in this way until all input vertices are processed. After termination $\mathcal{C}$ contains a decomposition of the input into simple cycles. The initialization of $\text{pos}$ takes $O(\vert V\vert)$ time. We process the input sequence by a single pass. Every vertex in the input sequence is put onto the stack at most once and removed from the stack at most once. The lookup in $\text{pos}$ takes $O(1)$. Therefore, the time complexity is $O(l)$.   

\begin{lemma} \label{prop:cycle_finding}
Given a list of edges $e_0,\ldots,e_l$ forming a cycle in $L$, we can find filling triangles in $O(l)$ time.
\end{lemma}

\begin{proof}
We now combine all the previous arguments in the Section. Let $\{e_0,\ldots,e_l\}$ be a potentially unordered list of edges that forms a cycle in $L$. Assuming a global map from edges to their boundary vertices, we can build a graph data structure, like an adjacency list, representing the cycle graph $C$ in linear time. We can then compute an Euler tour in $C$ in linear time. Given this Euler tour, we can decompose it into simple cycles in linear time, using the algorithm discussed above. For each simple cycle, we can directly read off the filling triangles in linear time by Lemma \ref{prop:filling_simple_cycles}. For that we need to know the longest edge but we can compute this during the cycle decomposition without overhead. Note that for the identification of vertices and triangles we do not have to check for each vertex what is the biggest power of two dividing it. We can just compute this identification while building the graph. Each triangle also corresponds to its longest edge $(a,b)$ and then the corresponding vertex is $\frac{a+b}{2}$. 
\end{proof}


\subsection{Proof of quasi-isomorphism}
\label{app:quasiiso_logpath}

This section is devoted to the proof of Theorem \ref{thm:quais_iso_logpath}. We restate Diagram \ref{diag:logpath_double_complex} for the convenience of the reader.
\begin{equation} \label{eq:diagram_logpath_app}
\begin{tikzcd} [column sep=large,row sep=large]
& \vdots \arrow[d] & \vdots \arrow[d] & \vdots \arrow[d] & \vdots \arrow[d] \\
0 & C_3 \arrow[l] \arrow[d,swap,"\partial_3"] & G_3 \arrow[l,swap,"\alpha_3"] \arrow[d,swap,"f_3^0"] \arrow[ddr,"h_3^0"{xshift=-10pt,yshift=10pt}] \arrow[dddrr,violet,bend left,"H_3^0"{xshift=-10pt,yshift=-14pt}] & R_3 \arrow[l,swap,"p_3^1"] \arrow[d,crossing over,"f_3^1"] \arrow[ddr,blue,"h_3^1"{xshift=-10pt,yshift=10pt}] & S_3 \arrow[l,swap,"p^2_3"] \arrow[d,red,"f_3^2"] & 0 \arrow[l] \\
0 & C_2 \arrow[l] \arrow[d,swap,"\partial_2"] & G_2 \arrow[l,swap,"\alpha_2"] \arrow[d,swap,"f_2^0"] \arrow[ddr,"h_2^0"{xshift=-10pt,yshift=10pt}] & R_2 \arrow[l,swap,"p_2^1"{xshift=3pt},crossing over] \arrow[d,"f_2^1"] \arrow[ddr,blue,"h_2^1"{xshift=-10pt,yshift=10pt}] & S_2 \arrow[l,swap,"p^2_2"{xshift=3pt},crossing over] \arrow[d,red,"f_2^2"] & 0 \arrow[l] \\
0 & C_1 \arrow[l] \arrow[d,swap,"\partial_1"] & G_1 \arrow[l,swap,"\alpha_1"] \arrow[d,swap,"f_1^0"] & R_1 \arrow[l,swap,"p_1^1"{xshift=3pt},crossing over] \arrow[d,"f_1^1"] & S_1 \arrow[l,swap,"p^2_1"{xshift=3pt},crossing over] \arrow[d,red,"f_1^2"] & 0 \arrow[l] \\
0 & C_0 \arrow[l] \arrow[d] & G_0 \arrow[l,swap,"\alpha_0"] \arrow[d] & R_0 \arrow[l,swap,"p_0^1"] \arrow[d] & S_0 \arrow[l,swap,"p^2_0"] \arrow [d]& 0 \arrow[l] \\
& 0 & 0 & 0 & 0
\end{tikzcd}
\end{equation} 

It is constructed from the input chain complex $C_\bullet$, induced by a simplicial bifiltration, in the following way: The $i$-th row is the sum of the log-path resolutions $\mathcal{L}_\sigma$ over all $i$-simplices and is, thus, a free resolution of $C_i$. The maps $f_i^\bullet$ are constructed such that the squares commute and are lifts of the boundary maps $\partial_i$ according to Lemma \ref{lem:fundamental_lemma}. Moreover, the maps $h^0_i,h^1_i,H^0_i$ in \eqref{eq:diagram_logpath_app} are constructed such that they satisfy
\begin{equation} \label{eq:map_properties}
\begin{aligned}
p_{i-2}^1\circ h^0_i=f^0_{i-1}\circ f^0_{i} \\
h^0_i\circ p^1_i + p^2_{i-2}\circ h^1_i = f^1_{i-1}\circ f^1_i \\ 
h^1_i\circ p_i^2=f^2_{i-1}\circ f^2_{i} \\
h^0_{i-1} \circ f^0_i + f^1_{i-2}\circ h^0_{i} =p^2_{i-3}\circ H_i^0 .
\end{aligned}
\end{equation}
We note again that the property $h^1_i\circ p_i^2=f^2_{i-1}\circ f^2_{i}$ follows from the fact that $p_{i-2}^2$ is a monomorphism, as in \ref{app:path_correctness}. The construction of these maps is based on the exactness of the rows in \eqref{eq:diagram_logpath_app} and

\kernels*

\begin{proof}
    By commutativity of Diagram \ref{eq:diagram_logpath_app}, it holds that
    \begin{equation*}
     \begin{split}
          p^1_i\circ (f^1_{i+1}\circ p^2_{i+1})=f_{i+1}^0\circ p_{i+1}^1\circ p_{i+1}^2=0 
     \end{split}
    \end{equation*}
    and
    \begin{equation*}
     \begin{split}
            p^1_{i}\circ (f^1_{i+1}\circ f^1_{i+2} + h^0_{i+2}\circ p^1_{i+2}) & = f^0_{i+1}\circ f^0_{i+2} \circ p^1_{i+2} + f_{i+1}^0\circ f_{i+2}^0 \circ p^1_{i+2} = 0
    \end{split}
    \end{equation*}
    and
    \begin{equation*}
    \begin{split}
    p^1_{i}\circ(h^0_{i+2}\circ f^0_{i+3} + f^1_{i+1} \circ h^0_{i+3}) &= f_{i+1}^0\circ f_{i+2}^0\circ f_{i+3}^0+f_{i+1}^0\circ p_{i+1}^1\circ h_{i+3}^0 \\ & = f_{i+1}^0\circ f_{i+2}^0\circ f_{i+3}^0+f_{i+1}^0\circ f_{i+2}^0\circ f_{i+3}^0=0 .
    \end{split}
    \end{equation*}
\end{proof}

\begin{remark}
On a high level, these maps can be understood in the following way. Again by Lemma \ref{lem:fundamental_lemma}, the composition $f_{i-1}^\bullet\circ f_i^\bullet$ lifts the zero morphism $\partial_{i-1}\circ\partial_i=0$ and is thus homotopic to zero. Hence, there exists a chain homotopy $h_i^\bullet$ such that $p_{i-2}^\bullet\circ h_i^\bullet+h_i^\bullet\circ p_i^\bullet=f_{i-1}^\bullet\circ f_i^\bullet$. Similarly, the composition $f_{i-2}^\bullet\circ f_{i-1}^\bullet\circ f_i^\bullet$ lifts $\partial_{i-2}\circ\partial_{i-1}\circ\partial_i=0$ and is therefore homotopic to zero. In this case, the compositions $f_{i-2}^\bullet\circ h_i^\bullet$ and $h_{i-1}^\bullet\circ f_i^\bullet$ constitute homotopies between zero and $f_{i-2}^\bullet\circ f_{i-1}^\bullet\circ f_i^\bullet$. Such homotopies are again unique up to a higher homotopy $H_i^\bullet$ such that $p_{i-3}^\bullet\circ H_i^\bullet+H_i^\bullet\circ p_i^\bullet=f_{i-2}^\bullet\circ h_i^\bullet+h_{i-1}^\bullet\circ f_i^\bullet$.  
\end{remark}

We now show that given Diagram \ref{eq:diagram_logpath_app}, with the properties discussed above, we obtain the following:
\begin{proposition}
The upper row of \eqref{eq:quasi_iso_3} is a chain complex and the vertical maps form a morphism of chain complexes.
\end{proposition}
\begin{equation}  \label{eq:quasi_iso_3}
\adjustbox{max width=0.92\displaywidth}{
\begin{tikzcd}[column sep=large,row sep=large,ampersand replacement=\&,every label/.append style = {font = \small}]
0 \&[-5pt] G_0 \arrow[l] \arrow[d,"\alpha_0"] \&[7pt] G_1\oplus R_0 \arrow{l}[swap]{\begin{pmatrix}f^0_1 & p_0^1\end{pmatrix}} \arrow{d}{\begin{pmatrix}\alpha_1 & 0\end{pmatrix}} \&[12pt] G_2\oplus R_1 \oplus S_0 \arrow{l}[swap]{\begin{pmatrix}f^0_2 & p_1^1 & 0 \\ h_2^0 & f_1^1 & p_0^2\end{pmatrix}} \arrow{d}{\begin{pmatrix}\alpha_2 & 0&0\end{pmatrix}} \&[12pt]  G_3\oplus R_2\oplus S_1 \arrow{d}{\begin{pmatrix}\alpha_3 & 0 & 0\end{pmatrix}} \arrow{l}[swap]{\begin{pmatrix}f^0_3 & p_2^1 & 0 \\ h_3^0 & f_2^1 & p_1^2 \\ H_3^0 & h_2^1 & f_1^2 \end{pmatrix}} \&[-5pt] \cdots \arrow[l] \\[10pt]
0 \& C_0 \arrow[l] \& C_1 \arrow[l,swap,"\partial_1"] \& C_2 \arrow[l,swap,"\partial_2"] \& C_3 \arrow[l,swap,"\partial_3"] \& \cdots \arrow[l]
\end{tikzcd}
}
\end{equation}

\begin{proof}
We take the composition of two consecutive differentials in \eqref{eq:quasi_iso_3}:
\begin{equation*} \small
\begin{aligned}
&\begin{pmatrix}f_i^0 & p_{i\minus 1}^1 & 0 \\ h_i^0 & f_{i\minus 1}^1 & p_{i\minus 2}^2 \\ H_i^0 & h_{i\minus 1}^1 & f_{i\minus 2}^2 \end{pmatrix}\circ\begin{pmatrix}f_{i+1}^0 & p_{i}^1 & 0 \\ h_{i+1}^0 & f_{i}^1 & p_{i\minus 1}^2 \\ H_{i+1}^0 & h_{i}^1 & f_{i\minus 1}^2 \end{pmatrix}= \\ &\begin{pmatrix}f_i^0\circ f_{i+1}^0+p_{i\minus 1}^1\circ h_{i+1}^0 & f_i^0\circ p_i^1+p_{i\minus 1}^1\circ f_i^1 & p_{i\minus 1}^1\circ p_{i\minus 1}^2 \\ h_i^0\circ f_{i+1}^0+f_{i\minus 1}^1\circ h_{i+1}^0+p_{i\minus 2}^2\circ H_{i+1}^0 & h_i^0\circ p_i^1+f_{i\minus 1}^1\circ f_i^1+p_{i\minus 2}^2\circ h_i^1 & f_{i\minus 1}^1\circ p_{i\minus 1}^2+p_{i\minus 2}^2\circ f_{i\minus 1}^2 \\ H_i^0\circ f_{i+1}^0 + h_{i\minus 1}^1\circ h_{i+1}^0 +  f_{i\minus 2}^2\circ H_{i+1}^0 & H_i^0\circ p_i^1 + h_{i\minus 1}^1\circ f_i^1 +  f_{i\minus 2}^2\circ h_i^1 & h_{i\minus 1}^1\circ p_{i\minus 1}^2 +  f_{i\minus 2}^2\circ f_{i\minus 1}^2 \end{pmatrix}.
\end{aligned}
\end{equation*}
All of the entries of the matrix representation of the composition are zero because of the exactness of the rows or commutativity in \eqref{eq:diagram_logpath_app} or by the properties \eqref{eq:map_properties}, except for $H_i^0\circ f_{i+1}^0 + h_{i\minus 1}^1\circ h_{i+1}^0 + f_{i\minus 2}^2\circ H_{i+1}^0$. But for this entry we obtain:
\begin{equation*}
\begin{aligned}
& p_{i\minus 3}^2\circ \left[H_i^0\circ f_{i+1}^0 + h_{i\minus 1}^1\circ h_{i+1}^0 + f_{i\minus 2}^2\circ H_{i+1}^0\right] \\ &=  \left[ h_{i\minus 1}^0\circ f_{i}^0 + f_{i\minus 2}^1\circ h_{i}^0 \right] \circ f_{i+1}^0  +  \left[ f_{i\minus 2}^1\circ f_{i\minus 1}^1+h_{i\minus 1}^0\circ p_{i\minus 1}^1 \right]\circ h_{i+1}^0 + \left[ f_{i\minus 2}^1\circ p_{i\minus 2}^2 \right]\circ H_{i+1}^0 \\ & = h_{i\minus 1}^0\circ f_{i}^0\circ f_{i+1}^0 + f_{i\minus 2}^1\circ h_{i}^0\circ f_{i+1}^0 +f_{i\minus 2}^1\circ f_{i\minus 1}^1\circ h_{i+1}^0 \\&+ h_{i\minus 1}^0\circ\left[ f_{i}^0\circ f_{i+1}^0\right]+f_{i\minus 2}^1\circ\left[ h_i^0\circ f_{i+1}^0+f_{i\minus 1}^1\circ h_{i+1}^0 \right]=0.
\end{aligned}
\end{equation*}
Since $p_{i\minus 3}^2$ is a monomorphism, also this entry vanishes. Thus, \eqref{eq:quasi_iso_3} is indeed a chain complex. Similarly, we check that
\begin{equation*}
\begin{aligned}
\begin{pmatrix}\alpha_{i\minus 1} & 0 & 0\end{pmatrix}\circ \begin{pmatrix}f_i^0 & p_{i\minus 1}^1 & 0 \\ h_i^0 & f_{i\minus 1}^1 & p_{i\minus 2}^2 \\ H_i^0 & h_{i\minus 1}^1 & f_{i\minus 2}^2 \end{pmatrix}=\begin{pmatrix}\alpha_{i\minus 1}\circ f_i^0 & \alpha_{i\minus 1}\circ p_{i\minus 1}^1 & 0 \end{pmatrix}=\partial_{i}\circ \begin{pmatrix} \alpha_i & 0 & 0 \end{pmatrix}
\end{aligned}
\end{equation*}
and, thus, the vertical maps form a morphism of chain complexes.
\end{proof}

We are now ready to show that the morphism in \eqref{eq:quasi_iso_3} is a quasi-isomorphism and, therefore, the upper row of \eqref{eq:quasi_iso_3} is a free resolution of $C_\bullet$.

\quasiisologpath*

\begin{proof}[Proof of Theorem \ref{thm:quais_iso_logpath}]
We show that the mapping cone, $\text{cone}(\alpha)$:
\begin{equation*}
\adjustbox{max width=\displaywidth}{
\begin{tikzcd}[ampersand replacement=\&,every label/.append style = {font = \small}]
0 \& C_0 \arrow[l] \&[5pt] C_1\oplus G_0 \arrow{l}[swap,yshift=5pt]{\begin{pmatrix} \partial_1 & \alpha_0\end{pmatrix}} \&[15pt] C_2\oplus G_1\oplus R_1 \arrow{l}[swap,yshift=5pt]{\begin{pmatrix} \partial_2 & \alpha_1 & 0 \\ 0 & f^0_1 & p_0^1\end{pmatrix}} \&[20pt] C_3\oplus G_2\oplus R_1\oplus S_0 \arrow{l}[swap,yshift=5pt]{\begin{pmatrix} \partial_3 & \alpha_2 & 0 & 0 \\ 0 & f^0_2 & p_1^1 & 0 \\ 0 & h_2^0 & f_1^1 & p_0^2 \end{pmatrix}} \&[20pt] C_4\oplus G_3\oplus R_2\oplus S_1 \arrow{l}[swap,yshift=5pt]{\begin{pmatrix} \partial_4 & \alpha_3 & 0 & 0 \\ 0 & f^0_3 & p_2^1 & 0 \\ 0 & h_3^0 & f_2^1 & p_1^2 \\ 0 & H_3^0 & h_2^1 & f_1^2 \end{pmatrix}} \& \cdots \arrow[l]  
\end{tikzcd}
}
\end{equation*}
of the chain morphism in \eqref{eq:quasi_iso_3} is acyclic. Assume
\begin{equation*}
\begin{pmatrix} \partial_{i} & \alpha_{i\minus 1} & 0 & 0 \\ 0 & f^0_{i\minus 1} & p_{i\minus 2}^1 & 0 \\ 0 & h_{i\minus 1}^0 & f_{i\minus 2}^1 & p_{i\minus 3}^2 \\ 0 & H_{i\minus 1}^0 & h_{i\minus 2}^1 & f_{i\minus 3}^2 \end{pmatrix}\begin{pmatrix}x_1\\x_2\\x_3\\x_4\end{pmatrix}=\begin{pmatrix} \partial_{i}(x_1)+\alpha_{i\minus 1}(x_2) \\ f^0_{i\minus 1}(x_2) + p_{i\minus 2}^1(x_3) \\ h_{i\minus 1}^0(x_2) + f_{i\minus 2}^1(x_3) + p_{i\minus 3}^2(x_4) \\ H_{i\minus 1}^0(x_2) + h_{i\minus 2}^1(x_3) + f_{i\minus 3}^2(x_4) \end{pmatrix}=\begin{pmatrix}0\\0\\0\\0\end{pmatrix}.
\end{equation*}
Since $\alpha_{i}$ is an epimorphism, there exists $y_1\in G_{i}$ such that $\alpha_{i}(y_1)=x_1$. Then $\partial_i\circ\alpha_{i}(y_1)=\alpha_{i\minus 1}\circ f_i^0(y_1)=\alpha_{i\minus 1}(x_2)$ and, thus, $\alpha_{i\minus 1}\big(f_i^0(y_1)+x_2\big)=0$. By exactness, there exists $y_2\in R_{i\minus 1}$ such that $p_{i\minus 1}^1(y_2)=f_i^0(y_1)+x_2$. Using this relation, we get:
\begin{equation*}
\begin{aligned}
f_{i\minus 1}^0\big(x_2\big)=f_{i\minus 1}^0\big(f_i^0(y_1)+p_{i\minus 1}^1(y_2)\big)= p_{i\minus 2}^1\circ h_{i}^0(y_1)+p_{i\minus 2}^1\circ f_{i\minus 1}^1(y_2)=p_{i\minus 2}^1(x_3) .
\end{aligned}
\end{equation*}
Hence, $p_{i\minus 2}^1\big(h_{i}^0(y_1)+f_{i\minus 1}^1(y_2)+x_3\big)=0$ and, by exactness, there exists $y_3\in S_{i\minus 2}$ such that $p_{i\minus 2}^2(y_3)=h_{i}^0(y_1)+f_{i\minus 1}^1(y_2)+x_3$. Using this relation, we obtain
\begin{equation*}
    \adjustbox{max width=\displaywidth}{%
        $\displaystyle
        \begin{aligned}
            h_{i-1}^0(x_2)+f_{i-2}^1(x_3)&=h_{i-1}^1\big(f_i^0(y_1)+p_{i-1}^1(y_2)\big)+f_{i-2}^1\big(h_{i}^0(y_1)+f_{i-1}^1(y_2)+p_{i-2}^2(y_3)\big) \\
            &= \left[h_{i-1}^1\circ f_i^0+f_{i-2}^1\circ h_{i}^0 \right](y_1)+\left[h_{i-1}^1\circ p_{i-1}^1+f_{i-2}^1\circ f_{i-1}^1 \right](y_2)+f_{i-2}^1\circ p_{i-2}^2(y_3) \\
            &= p_{i-3}^2\circ H_i^0(y_1)+p_{i-3}^2\circ h_{i-1}^1(y_2)+p_{i-3}^2\circ f_{i-2}^2(y_3)=p_{i-3}^2(x_4).
        \end{aligned}
        $
    }
\end{equation*}
Hence, $p_{i\minus 3}^2\big(H_i^0(y_1)+h_{i\minus 1}^1(y_2)+f_{i\minus 2}^2(y_3)+x_4\big)=0$ and, since $p_{i\minus 3}^2$ is a monomorphism $x_4=H_i^0(y_1)+h_{i\minus 1}^1(y_2)+f_{i\minus 2}^2(y_3)$. We conclude that 
\begin{equation*}
\begin{pmatrix} \partial_{i+1} & \alpha_{i} & 0 & 0 \\ 0 & f^0_{i} & p_{i\minus 1}^1 & 0 \\ 0 & h_{i}^0 & f_{i\minus 1}^1 & p_{i\minus 2}^2 \\ 0 & H_{i}^0 & h_{i\minus 1}^1 & f_{i\minus 2}^2 \end{pmatrix}\begin{pmatrix}0\\y_1\\y_2\\y_3\end{pmatrix}=\begin{pmatrix}\alpha_i(y_1)\\f_i^0(y_1)+p_{i\minus 1}^1(y_2)\\h_{i}^0(y_1)+f_{i\minus 1}^1(y_2)+p_{i\minus 2}^2(y_3)\\H_i^0(y_1)+h_{i\minus 1}^1(y_2)+f_{i\minus 2}^2(y_3)\end{pmatrix}=\begin{pmatrix}x_1\\x_2\\x_3\\x_4\end{pmatrix}
\end{equation*}
and, therefore, the mapping cone is exact. This implies that \eqref{eq:quasi_iso_3} is a quasi-isomorphism (see Corollary $10.41$ in \cite{rotman}).
\end{proof}

\subsection{Complexity and correctness for the log-path algorithm} \label{app:log_path_algo}

\logpathalgo*

\begin{proof}[Proof of Theorem \ref{thm:logpath_theorem}]
The correctness of the algorithm is a direct consequence of Theorem \ref{thm:quais_iso_logpath}. In Step 1, we build $\mathcal{L}_\sigma$ for all input simplices. This requires iterating through all simplex grades once and adding $O(n)$ relations and syzygies, which costs $O(n)$ time. In Step 2, we have to go over all generators $g_\sigma^x$ and find boundary generators $g_\tau^y$ with $y\leq x$ for all facets $\tau$ of $\sigma$. Such a $g_\tau^y$ can be found in $O(\log n_\tau)$ time and, thus, the image of each $g_\sigma^x$ can be determined in $O(d\log k)$ time. Hence, overall this step takes $O(nd\log k)$ time. For Step 3, we note that all involved matrices have $O(n)$ columns and, by construction, each column of $f^0_i$ has at most $d$ non-zero entries while each column of $p^1_i$ has exactly two non-zero entries. Therefore, we can compute the sparse matrix product $f^0_i\circ p^1_i$ and $f^0_{i-1}\circ f^0_i$ in $O(nd^2)$ time. The columns of these products contain at most $d^2$ pairs of generators $g^x_\sigma$, $g_\sigma^{x'}$ which have to be connected by a shortest monotone path. Such a path can be found in $O(\log n_\sigma)$ time by Lemma \ref{lem:log-path-lemma}. Thus, overall Step 3 takes $O(nd^2\log k)$ time. For Step 4, we note that the matrix column sparsity of $f^1_i$ and $h^0_i$ is $O(d\log k)$ and $O(d^2\log k)$, respectively, while in $p^2_i$ every column has exactly three non-zero entries. Hence, we can compute all sparse matrix products in $O(nd^3\log^2 k)$ and the worst case column sparsity of the results is $O(d^3\log^2k)$. The resulting matrices contain cycles of length $O(d^3\log^2k)$ which can be filled by triangles in $O(d^3\log^2k)$ time by Lemma \ref{ref:cycle_lemma}. Therefore, overall, Step 4 takes $O(nd^3\log^2k)$ time. Since we assume the dimension $d$ is constant, we obtain an overall time complexity of $O(n\log^2 k)$. By the discussion of the column sparsity of the involved matrices, we also obtain that the description size of the output is $O(n\log^2 k)$.
\end{proof}

\end{document}